\documentclass[11pt]{amsart}
\usepackage[utf8]{inputenc}
\usepackage{amssymb,amsmath}
\usepackage{graphicx}
\usepackage{color}
\usepackage{tikz}
\usepackage{pgfplots}
\usepackage{enumitem}
\usepackage{algorithm}
\usepackage{setspace}
\usepackage{algpseudocode}
\usepackage{mathtools}
\usepackage{xfrac}

\usepackage{accents}
\usepackage{multicol} 
\usepackage{multirow}
\usepackage{soul}
\usepackage{subcaption}
\usepackage{booktabs}
\usepackage{array}
\usepackage{comment}
\newcolumntype{L}[1]{>{\raggedright\let\newline\\\arraybackslash\hspace{0pt}}m{#1}}
\newcolumntype{C}[1]{>{\centering\let\newline\\\arraybackslash\hspace{0pt}}m{#1}}
\newcolumntype{R}[1]{>{\raggedleft\let\newline\\\arraybackslash\hspace{0pt}}m{#1}}
\usepackage{siunitx}
\usepackage{bbold}

\usepackage[hidelinks]{hyperref}
\usepackage[nameinlink,noabbrev]{cleveref}
\newtheorem{theorem}{Theorem}
\newtheorem{proposition}[theorem]{Proposition}

\theoremstyle{definition}
\newtheorem{definition}[theorem]{Definition}

\theoremstyle{lemma}
\newtheorem{lemma}[theorem]{Lemma}

\theoremstyle{remark}
\newtheorem{remark}[theorem]{Remark}

\newtheorem{assumption}[theorem]{Assumption}

\Crefname{assumption}{Assumption}{Assumptions}
\numberwithin{theorem}{section}
\numberwithin{equation}{section}
\numberwithin{table}{section}
\numberwithin{figure}{section}
\definecolor{myBlue}{RGB}{30,144,255} % dodger blue
\definecolor{myGreen}{RGB}{69,169,0} % chatreuse
\definecolor{myRed}{RGB}{165,12,42} 
\definecolor{myOrange}{RGB}{225,92,22} 
\definecolor{mycolor0}{rgb}{0.12156862745098,0.466666666666667,0.705882352941177} % orange
\definecolor{mycolor1}{rgb}{0.00000,0.44700,0.74100}% blau
\definecolor{mycolor2}{rgb}{0.85000,0.32500,0.09800}% rot/orange
\definecolor{mycolor3}{rgb}{0.49400,0.18400,0.55600}% lila
\definecolor{mycolor4}{rgb}{0.92900,0.69400,0.12500}% gelb
\definecolor{mycolor5}{rgb}{0.46600,0.67400,0.18800}% gruen
\definecolor{mycolor6}{rgb}{0.30100,0.74500,0.93300}% hellblau
\definecolor{mycolor7}{rgb}{0.63500,0.07800,0.18400}% pflaume

\newcommand{\delete}[1]{ }

\def\R{\mathbb{R}}

\DeclareMathOperator{\Diag}{Diag}

\newcommand{\calO}{\ensuremath{\mathcal{O}} }

\newcommand{\calS}{\ensuremath{\mathcal{S}} }

\newcommand{\f}{\ensuremath{f}}
\newcommand{\g}{\ensuremath{g}}
\newcommand{\Amat}{\mathbf{A}}
\newcommand{\Bmat}{\mathbf{B}}
\newcommand{\Cmat}{\mathbf{C}}
\newcommand{\Dmat}{\mathbf{D}}
\newcommand{\Imat}{\mathbf{I}}

\newcommand{\Smat}{\mathbf{S}}
\newcommand{\Tmat}{\mathbf{T}}

\newcommand{\tol}{\operatorname{TOL}}
\newcommand{\Cfg}{\ensuremath{C_\text{rhs}}}
\newcommand{\Ctau}{\ensuremath{{\Cmat_\tau}}}
\newcommand{\Ctauinv}{\ensuremath{{\Cmat^{-1}_\tau}}}
\newcommand{\BDF}{\ensuremath{\text{BDF}}}

\DeclareFontFamily{U}{matha}{\hyphenchar\font45}
\DeclareFontShape{U}{matha}{m}{n}{
	<-6> matha5 <6-7> matha6 <7-8> matha7
	<8-9> matha8 <9-10> matha9
	<10-12> matha10 <12-> matha12
}{}
\DeclareSymbolFont{matha}{U}{matha}{m}{n}

\DeclareFontFamily{U}{mathx}{\hyphenchar\font45}
\DeclareFontShape{U}{mathx}{m}{n}{
	<-6> mathx5 <6-7> mathx6 <7-8> mathx7
	<8-9> mathx8 <9-10> mathx9
	<10-12> mathx10 <12-> mathx12
}{}
\DeclareSymbolFont{mathx}{U}{mathx}{m}{n}

\DeclareMathDelimiter{\vvvert} {0}{matha}{"7E}{mathx}{"17}%

\DeclareMathOperator{\newton}{N}
\DeclareMathOperator{\meter}{m}

\DeclareMathOperator{\second}{s}
\DeclareMathOperator{\minute}{min}

\allowdisplaybreaks
\begin{document}
\title[Second-Order Iterative Time Integration for Linear Poroelasticity]{A Second-Order Iterative Time Integration Scheme\\ for Linear Poroelasticity}
\author[]{R.~Altmann$^\dagger$, M.~Deiml$^\ddagger$}
\address{${}^{\dagger}$ Institute of Analysis and Numerics, Otto von Guericke University Magdeburg, Universit\"atsplatz 2, 39106 Magdeburg, Germany}
\address{${}^{\ddagger}$ Institute of Mathematics, University of Augsburg, Universit\"atsstr.~12a, 86159 Augsburg, Germany}
\email{robert.altmann@ovgu.de, matthias.deiml@uni-a.de}
%\thanks{}
%
\date{\today}
%
%
%=============================================================================
%=========  Abstract
%=============================================================================
\begin{abstract}
We propose a novel time stepping method for linear poroelasticity by extending a recent iterative decoupling approach to the second-order case. This results in a two-step scheme with an inner iteration and a relaxation step. We prove second-order convergence for a prescribed number of inner iteration steps, only depending on the coupling strength of the elastic and the flow equation. The efficiency of the scheme is illustrated by a number of numerical experiments, including a simulation of three-dimensional brain tissue. 
\end{abstract}
%
%% TODOs bzw. was man noch einfügen könnte:
%  
% - Beweis, dass Remark 4.3 zu beta=5 führt (muss nicht ins Paper aber sollte existieren)
% - Beweise checken, insb Ende Th. 4.5
%
% - numerische Experimente:
%   - a) fixed-stress second order einbauen
%   - b) omega variieren --> Plot omega/Laufzeit für vers Verfahren
%   - c) nicht-lineares Beispiel
%
%
%=============================================================================
%=========  Title / Contents
%=============================================================================
\maketitle
%\setcounter{tocdepth}{3}
%\tableofcontents
%
%
%65M12 = Stability and convergence of numerical methods PDEs
%65M20 = Method of lines
%65L80 = Methods for differential-algebraic equations
%76S05 = Flows in porous media; filtration
%65M60 = Finite elements, Rayleigh-Ritz and Galerkin methods, finite methods
% 
%
%=============================================================================
%=========  Introduction
%=============================================================================
\section{Introduction}

% poroelasticity and applications
Poroelasticity is a popular model applied in a variety of fields such as in geomechanics for porous rocks~\cite{DetC93} or in medicine for biological tissue~\cite{JuCLT20}. The model was first introduced in \cite{Bio41} and later reformulated with the solid displacement and fluid pressure as variables~\cite{Sho00}, nowadays known as two-field formulation. This leads to an elliptic differential equation describing the balance of stress as well as a parabolic equation describing the conservation of mass, which together form an elliptic--parabolic system.

% spatial discretization
For the spatial discretization, a number of finite element methods have been considered, both conforming and non-conforming \cite{MurL92, PhiW07a, PhiW08, LeeMW17, HonK18}. Within this work, we assume an existing finite element discretization in space and focus on the resulting differential--algebraic equation. For the respective spatial error analysis, we refer to the aforementioned sources. We would also like to mention the possibility of space--time approaches as considered in \cite{BauRK17, ArfS22}. These, however, are not compatible with the proposed time steppe scheme.

% implicit time-descretizations
A purely explicit time-discretization of the poroelasticity equations lacks the necessary stability. As such, fully implicit schemes like the \textit{implicit Euler method} are typically used~\cite{ErnM09}. This approach admits unconditional stability, but due to the large coupled nature of the resulting linear system, does not scale well with finer grid sizes.
% decoupling time-descretizations
Decoupling schemes, in contrast, separate the two equations and solve them successively, often leading to a lower computational complexity, while enabling the use of standard preconditioners for the elasticity and the flow problem~\cite{LeeMW17}. Moreover, the resulting left-hand side matrices are symmetric positive definite such that more efficient linear solvers can be applied.
% iterative schemes
The decoupling is usually implemented as an inner iteration like in the \textit{fixed-strain}, \textit{fixed-stress}, \textit{drained}, and \textit{undrained} approaches \cite{KimTJ11a, KimTJ11b}. There one can observe analytically, that unconditional stability is only achieved through relaxation of the inner iteration. Otherwise, the approximation only converges under a certain \textit{weak coupling condition}.
% Semi-explicit euler
It was however found that in presence of this condition convergence is already achieved through a single inner iteration step, leading to the \textit{semi-explicit Euler scheme}~\cite{AltMU21}. This approach was later extended to construct a related second-order method based on the BDF-$2$ scheme~\cite{AltMU24}.
% Iterative semi-explicit method
More recently, it was analyzed, whether a small number of inner iterations would relax the restriction regarding the coupling strength~\cite{AltD24}. This leads to another iterative approach (similar to the undrained scheme but with a different kind of relaxation) and it was shown that for any coupling strength an appropriate number of inner iterations can be determined, which then guarantees first-order convergence.

% outline
In this paper, we extend the last-mentioned iterative approach to a second-order scheme; see Section~\ref{sec:methods}. This means that -- depending on the coupling strength -- the number of inner iteration steps is fixed a priori. With this, the scheme can be applied to any poroelastic problem. For the convergence analysis, which we present in Section~\ref{sec:convergence}, we pick up the strategy presented in~\cite{AltD24}. We prove convergence of order $7/4$ for the general case and order $2$ if an additional assumption on the initial data is satisfied. The latter, however, seems more like a technical assumption and does not occur in the numerical experiments of Section~\ref{sec:numerics}. Therein, we study the sharpness of the bound on the needed number of inner iterations and consider a three-dimensional real-world brain tissue example. 
%
%
%=============================================================================
%=========  Model Problem 
%=============================================================================
\section{Preliminaries}\label{sec:prelim}
In this preliminary section, we introduce the (semi-discrete) equations of linear poroelasticity, including the coupling strength~$\omega$ which plays a key role in this paper. A spatial discretization of the coupled system introduced in~\cite{Bio41,Sho00} leads to a system of the form 
\begin{align}
	\label{eq:semiDiscretePoro}
	\begin{bmatrix} 0 & 0 \\ \Dmat & \Cmat \end{bmatrix}
	\begin{bmatrix} \dot u \\ \dot p \end{bmatrix}
	= 
	\begin{bmatrix} -\Amat & \phantom{-}\Dmat^T \\ 0 & -\Bmat \end{bmatrix}
	\begin{bmatrix} u \\ p \end{bmatrix}
	+ 
	\begin{bmatrix} \f \\ \g \end{bmatrix}
\end{align}
with initial conditions 
\[
	u(0) = u^0, \qquad
	p(0) = p^0. 
\]
Due to the differential--algebraic structure of system~\eqref{eq:semiDiscretePoro}, the initial data needs to satisfy the consistency condition $\Amat u^0 = \f(0) + \Dmat^T p^0$; see~\cite{AltMU21c} for more details. The (homogeneous) boundary conditions of the PDE model, on the other hand, are assumed to be included in the spatial discretization. For more details, we exemplary mention~\cite{MurL92,PhiW07a,HonK18}; see also the references therein.
%
% meaning of matrices
In general, $\Amat\in \R^{n_u, n_u}$ is the stiffness matrix which describes the elasticity in the mechanical part of the problem, whereas $\Bmat\in \R^{n_p, n_p}$ equals the stiffness matrix which is related to the fluid flow. Moreover, $\Cmat\in \R^{n_p, n_p}$ is a mass matrix which scales with the Biot modulus describing the compressibility of the fluid under pressure. 
% general assumptions
In what follows, we are not restricted to one particular discretization and only consider a number of structural assumptions for the matrices. 
\begin{assumption}[Spatial discretization]
\label{ass:matrices}
The matrices $\Amat$, $\Bmat$, and $\Cmat$ are symmetric and positive definite. Moreover, $\Dmat\in \R^{n_p, n_u}$ is of full (row) rank. 
\end{assumption}
Given Assumption~\ref{ass:matrices}, we denote by $c_\Amat$, $c_\Bmat$, and $c_\Cmat$ the smallest eigenvalues of~$\Amat$, $\Bmat$, and $\Cmat$, respectively. Obviously, these values are positive and real. We further denote the spectral norms by
\[
	C_\Amat \coloneqq \|\Amat\|_2,\qquad
	C_\Bmat \coloneqq \|\Bmat\|_2,\qquad
	C_\Cmat \coloneqq \|\Cmat\|_2,\qquad
	C_\Dmat \coloneqq \max\{\|\Dmat\|_2, \|\Dmat^T\|_2\}.
\]
For $\Amat$, $\Bmat$, and $\Cmat$, these numbers coincide with their respective largest eigenvalues. These three matrices also induce norms, namely 
\[
	\|v\|_\Amat \coloneqq \sqrt{v^T\Amat v}, \qquad
	\|q\|_\Bmat \coloneqq \sqrt{q^T\Bmat q}, \qquad
	\|q\|_\Cmat \coloneqq \sqrt{q^T\Cmat q}.
\]
Moreover, we use the notation $\Ctau \coloneqq \Cmat + \tfrac23\tau\Bmat$ with the time step size $\tau$ being introduced later on. This again induces a norm, 
\[
	\|q\|_\Ctau \coloneqq \sqrt{q^T\Ctau q}.
\]
We are now interested in the operator norm of $\Dmat$ under the problem dependent norms $\|\bullet\|_\Amat$ and $\|\bullet\|_\Ctau$. This relation describes the coupling strength of the elastic and flow equation. 
\begin{definition}\label{def:coupling-strength}
The \textit{coupling strength} is defined as
\[ 
	\omega 
	\coloneqq \bigg(\inf_{q \ne 0} \frac{\|\Dmat q\|_\Amat}{\|q\|_\Ctau}\bigg)^2 
\]
and satisfies under Assumption~\ref{ass:matrices} the estimate
\[
	\omega 
	\le \frac{C_\Dmat^2}{c_\Amat(c_\Cmat + \tfrac23\tau c_\Bmat)} 
	\le \frac{C_\Dmat^2}{c_\Amat c_\Cmat}.
\]
\end{definition}

In the considered poroelastic setting, a conforming spatial discretization yields  
\[ 
	c_\Amat \le \mu, \qquad 
	c_\Cmat = 1/M, \qquad 
	C_\Dmat \le \alpha\sqrt{m} 
\]
with the Lam{\'e} coefficient~$\mu$, the Biot modulus~$M$, and the Biot--Willis fluid--solid coupling coefficient~$\alpha$. Using Definition~\ref{def:coupling-strength}, we obtain $\omega \le m\,\alpha^2M/\mu$. Following the calculations in~\cite{AltD24}, this bound can be improved to $\omega \le \alpha^2M/(\mu + \lambda)$ with $\lambda$ denoting the second Lam{\'e} coefficient. In this paper, we are interested in applications with moderate $\omega$. In geomechanics, for example, $\omega$ is often in the range of $1$. For a list of examples, we refer to~\cite{DetC93}; see also~\cite{AltMU24}. 

In preparation for the upcoming analysis, we introduce relevant matrices and corresponding norm estimates. 
\begin{lemma}[Matrix norm estimates]\label{lem:matrices}
Consider $0 < \gamma < 1$.Then the $(n_p \times n_p)$-matrices  
\begin{align*}
	\Ctau &\coloneqq \Cmat + \tfrac23\tau\Bmat, &
	\Tmat &\coloneqq -\Ctauinv\Dmat \Amat^{-1}\Dmat^T, \\
	\Smat &\coloneqq \gamma\, \Tmat + (1 - \gamma)\, \Imat, &
	\tilde \Smat &\coloneqq \Smat^{K-1} + \gamma \sum\nolimits_{k=0}^{K-2} \Smat^k, 
\end{align*}
where $\Imat$ denotes the identity matrix, satisfy the spectral bounds 
\begin{align*}
	\|\Ctauinv \Cmat x\|_\Ctau &\le \|x\|_\Ctau	&
	\|\Tmat x\|_\Ctau &\le \omega\, \|x\|_\Ctau, \\
	\|\Smat x\|_\Ctau &\le (1 - \gamma)\, \|x\|_\Ctau,	&
	\|\tilde \Smat x\|_\Ctau &\le \|x\|_\Ctau.
\end{align*}
\end{lemma} 
\begin{proof}
% \Ctauinv\Cmat
To see $\|\Ctauinv \Cmat x\|_\Ctau \le \|x\|_\Ctau$, we first observe that 
\begin{align*}
	\|x\|_\Ctau^2
	= \|\Ctauinv\Ctau x\|_\Ctau^2 
	&= \|\mathbf{C}_\tau^{-1/2}(\Cmat + \tfrac23\tau \Bmat)x\|^2 \\
	&= \|\mathbf{C}_\tau^{-1/2}\Cmat x\|^2 + 2x^T(\tfrac23\tau \Bmat)\Ctauinv\Cmat x + \tfrac49\tau^2\, \|\mathbf{C}_\tau^{-1/2}\Bmat x\|^2 \\
	&= \|\Ctauinv\Cmat x\|_\Ctau^2 + \tfrac43\tau\, x^T\Bmat\Ctauinv\Cmat x + \tfrac49\tau^2\, \|\Ctauinv\Bmat x\|_\Ctau^2.
\end{align*}
Note that the last summand is non-negative. Hence, it remains to show that $\tau\Bmat\Ctauinv\Cmat$ is positive semi-definite. Indeed, we have 
\[ 
	\big( \tau\Bmat\Ctauinv\Cmat \big)^{-1} 
	= \Cmat^{-1}\big(\Cmat + \tfrac23\tau\Bmat \big)(\tau\Bmat)^{-1} 
	= (\tau\Bmat)^{-1} + \tfrac23\,\Cmat^{-1},  
\]
which is even positive definite, since the inverse as well as the sum of positive matrices are positive. 
% \Tmat
For the estimate including $\Tmat$, we use Definition~\ref{def:coupling-strength}, which yields for the spectral radius $\rho(\Tmat) \le \omega$. This, in turn, implies the stated estimate.  
% \Smat
Since the product of symmetric positive definite matrices stays positive, the eigenvalues of $\Tmat$ are negative. Hence, the spectrum of~$\gamma\, \Tmat$ lies in the interval $[-\gamma\omega,0]$. The matrix $\Smat$ shifts this by $1 - \gamma$ in positive direction, leading to a spectrum in $[-(1-\gamma),1-\gamma]$. This implies $\rho(\Smat)\le 1-\gamma$. 
% \tilde S
Finally, to estimate the norm of $\tilde \Smat$ we use
\[ 
	\|\tilde \Smat x\|_\Ctau 
	\le \Big((1 - \gamma)^{K-1} + \gamma\, \sum\nolimits_{k=0}^{K-2} (1 - \gamma)^k\Big)\|x\|_\Ctau. 
\]
Applying the geometric series, we observe that the term in parenthesis equals 
\[ 
	(1 - \gamma)^{K-1} + \gamma\, \sum\nolimits_{k=0}^{K-2} (1 - \gamma)^k 
	= (1 - \gamma)^{K-1} + \gamma\, \frac{1 - (1 - \gamma)^{K-1}}{1 - (1 - \gamma)} 
	= 1. \qedhere
\]
\end{proof}
%
%\begin{remark}
%The matrix $\tilde \Smat$ can also be written as~$\tilde \Smat = (\gamma\Tmat + \mathbf{I})^{K-1}$. As such, it is positive semi-definite if $K$ is odd. 
%\end{remark}
%
%
%=============================================================================
%=========  Numerical Schemes
%=============================================================================
\section{Time Stepping Schemes of Second Order}\label{sec:methods}
We consider an equidistant decomposition of the time interval~$[0,T]$ with step size~$\tau$, resulting in discrete time points~$t^n = n\tau$ for $0 \le n \le N$. Moreover, we recall the notion $\Ctau = \Cmat + \tfrac23\tau\Bmat$, which will be used throughout this section. 
Approximations of~$u(t^n)$ and~$p(t^n)$ are denoted by~$u^n$ and $p^n$, respectively. In the same manner, we write~$\f^{n} \coloneqq \f(t^n) \in \R^{n_u}$ and~$\g^{n} \coloneqq \g(t^n) \in \R^{n_p}$ for the point evaluations of the right-hand sides. 
%
%
%=============================================================================
\subsection{Implicit discretization of second order}
A direct application of BDF-$2$ to the semi-discrete system~\eqref{eq:semiDiscretePoro} yields the two-step scheme  
\begin{align}
\label{eq:BDFimplicit}	
	\begin{bmatrix} \Amat & -\Dmat^T \\ 3 \Dmat & 3\Ctau \end{bmatrix}
	\begin{bmatrix} u^{n+2} \\ p^{n+2} \end{bmatrix}
	=
	\begin{bmatrix} f^{n+2} \\ 2\tau g^{n+2} + \Dmat(4u^{n+1} - u^n) + \Cmat\,(4p^{n+1} - p^n) \end{bmatrix}.
\end{align}
Given Assumption~\ref{ass:matrices}, this scheme is well-posed as the leading matrix on the left-hand side is invertible. Moreover, it is second-order convergent as shown in~\cite[Th.~5.24]{KunM06}. 
This fully implicit discretization, however, suffers from the computational effort needed in each time step. Especially for three-dimensional applications, the resulting linear system is quite challenging due to the actual size of the system and the lack of appropriate preconditioners~\cite{LeeMW17}. Hence, we aim for decoupled schemes which allow the application of well-established preconditioners for the matrices~$\Amat$, $\Bmat$, and $\Cmat$. 
%
%
%=============================================================================
\subsection{Semi-explicit discretization of second order}
One approach for the decoupling of the system is presented in~\cite{AltMU24} and reads 
\begin{align}
\label{eq:BDFsemiExpl} 
	\begin{bmatrix} \Amat & \\ 3 \Dmat & 3 \Ctau \end{bmatrix}
	\begin{bmatrix} u^{n+2} \\ p^{n+2} \end{bmatrix}
	=
	\begin{bmatrix} f^{n+2} + \Dmat^T (2p^{n+1} - p^n) \\ 2\tau g^{n+2} + \Dmat(4u^{n+1} - u^n) + \Cmat\,(4p^{n+1} - p^n) \end{bmatrix}.
\end{align}
%
%\begin{subequations} 
%\label{eq:second-order-semi}
%\begin{align}
%	\Amat u^{n+2} - \Dmat^T (2p^{n+1} - p^n) 
%	&= f^{n+2}, \\
%	\Dmat(3u^{n+2} - 4u^{n+1} + u^n) + \Cmat\,(3p^{n+2} - 4p^{n+1} + p^n) + 2\tau \Bmat p^{n+2} 
%	&= 2\tau g^{n+2}
%\end{align}
%\end{subequations}
%
This scheme results from applying the BDF-$2$ scheme to the delay system
\begin{align*}
	\Amat u - \Dmat^T(2\Delta_\tau p - \Delta_{2\tau} p) 
	&= f, \\
	\Dmat\dot{u} + \Cmat\dot{p} + \Bmat p 
	&= g,  
\end{align*}
where $\Delta_\tau p(t) \coloneqq p(t-\tau)$. Therein, the delay term is motivated by applying discrete derivatives to the Taylor expansion~$p(t) \approx \Delta_\tau p + \tau\Delta_\tau \dot{p}$. Scheme~\eqref{eq:BDFsemiExpl} has the advantage that the two equations can be solved sequentially, i.e., we first solve for~$u^{n+2}$ (by inversion of~$\Amat$) and then for~$p^{n+2}$ (by inversion of~$\Ctau$). We would like to emphasize that there exist well-established preconditioners for both subsystems. 

On the other hand, convergence is only guaranteed for~$\omega < 1/3$; see~\cite{AltMU24}. For larger coupling constants, the scheme gets unstable. Aim of this paper is to include an inner iteration which loosens up the restrictive restriction on the coupling strength. 
%
%\begin{subequations} \label{eq:approach}
%\begin{align}
%	\Amat u - \Dmat^T(\Delta_\tau p + \tau\Delta_\tau \dot{p}) 
%	&= f, \\
%	\Dmat \dot{u} + \Cmat\dot{p} + \Bmat p 
%	&= g.
%\end{align}
%\end{subequations}
%
%
%=============================================================================
\subsection{Second-order fixed stress}\label{sec:methods:fixedstress}
The fixed stress scheme~\cite{KimTJ11b, WheG07} is one of the most well-known iterative methods for poroelasticity. It is based on a fixed-point iteration which is motivated by physics and provides (in its original form) a first-order scheme as long as the number iteration steps is sufficiently large. In~\cite{AltMU23ppt}, this scheme was extended to higher order.  
% construction 
Based on the fixed-point equation 
\begin{align*}
	\Dmat ({\dot u}_{k+1} - {\dot u}_{k})
	= L ({\dot p}_{k+1} - {\dot p}_{k}) 
\end{align*}
with some stabilization parameter~$L > 0$ and subscript $k$ denoting the (inner) iteration number, system~\eqref{eq:semiDiscretePoro} decouples into  
\begin{align} \label{eq:fixed-stress}
	\begin{bmatrix} 0 & 0 \\ 0 & \Cmat + L\Imat \end{bmatrix}
	\begin{bmatrix} {\dot u}_{k+1} \\ {\dot p}_{k+1} \end{bmatrix}
	= 
	\begin{bmatrix} 0 & 0 \\ -\Dmat & L\Imat \end{bmatrix}
	\begin{bmatrix} {\dot u}_{k} \\ {\dot p}_{k} \end{bmatrix}
	+
	\begin{bmatrix} -\Amat & \phantom{-}\Dmat^T \\ 0 & -\Bmat \end{bmatrix}
	\begin{bmatrix} u_{k+1} \\ p_{k+1} \end{bmatrix}
	+ 
	\begin{bmatrix} \f \\ \g \end{bmatrix}.
\end{align}
Now, BDF-$2$ is applied, leading to the following scheme: given approximations~$(u^{n}, p^{n})$, $(u^{n+1}, p^{n+1})$ and initial guesses~$u^{n+2}_0$, $p^{n+2}_0$ (e.g.~given by the approximation at the previous time step), compute 
\begin{align*}
	&\begin{bmatrix} \Amat & -\Dmat^T \\ 0 & 3\Ctau + 3 L\Imat \end{bmatrix}
	\begin{bmatrix} u^{n+2}_{k+1} \\ p^{n+2}_{k+1} \end{bmatrix}\\
	&\hspace{1.8cm}= 
	\begin{bmatrix} 0 & 0 \\ -3 \Dmat & 3 L\Imat \end{bmatrix}
	\begin{bmatrix} u^{n+2}_{k} \\ p^{n+2}_{k} \end{bmatrix}
	+
	\begin{bmatrix} 0 & 0 \\ 4\Dmat & 4\Cmat \end{bmatrix}
	\begin{bmatrix} u^{n+1} \\ p^{n+1} \end{bmatrix}
	-
	\begin{bmatrix} 0 & 0 \\ \Dmat & \Cmat \end{bmatrix}
	\begin{bmatrix} u^{n} \\ p^{n} \end{bmatrix}
	+ 
	\begin{bmatrix} \f^{n+2} \\ 2\tau\g^{n+2} \end{bmatrix}.
\end{align*}
Note that this again decouples the system. As stopping criterion, one considers the difference of consecutive iterates, i.e., the inner iteration is stopped if 
% !! im paper (kontinuierlich) nutzt man die $\V$ und $\cHV$ Norm 
%
\[
	\Vert u^{n+2}_{k+1} - u^{n+2}_{k}\Vert^{2}_{\Amat} + \Vert p^{n+2}_{k+1} - u^{n+2}_{k} \Vert^{2}_{\Cmat}
	\le \tol^{2}
\]
for a prescribed tolerance~$\tol$. The introduced iteration is contractive, which guarantees its convergence. Following the convergence result of~\cite{AltMU23ppt}, the tolerance has to be chosen as $\tol \lesssim C\,\tau^3$ in order to obtain a second-order time stepping scheme. 

Although such a stopping criterion quantifies the number of needed iteration steps per time step, this number may still be large and is not known a priori. In the following, we want to construct an iterative scheme with a fixed number of inner iteration steps. 
%
%
%=============================================================================
\subsection{A first iterative approach}
An application of the matrix splitting 
\begin{align}
\label{eq:matrixSplitting}
	\begin{bmatrix} \Amat & -\Dmat^T \\ 3 \Dmat & 3 \Ctau \end{bmatrix}
	= \begin{bmatrix} \Amat & 0 \\ 3 \Dmat & 3 \Ctau \end{bmatrix}
	- \begin{bmatrix} 0 & \Dmat^T \\ 0 & 0 \end{bmatrix}
\end{align}
to the BDF discretization~\eqref{eq:BDFimplicit} would lead to the iterative scheme
\begin{equation*}
	\begin{bmatrix} \Amat & 0 \\ 3 \Dmat & 3 \Ctau \end{bmatrix}
	\begin{bmatrix} u_{k+1}^{n+2} \\ p_{k+1}^{n+2} \end{bmatrix}
	=
	\begin{bmatrix} 0 & \Dmat^T \\ 0 & 0 \end{bmatrix}
	\begin{bmatrix} u_{k}^{n+2} \\ p_{k}^{n+2} \end{bmatrix}
	+
	\begin{bmatrix} f^{n+2} \\ 2\tau g^{n+2} + \Dmat(4u^{n+1} - u^n) + \Cmat\,(4p^{n+1} - p^n) \end{bmatrix},
\end{equation*}
where we initialize $u_{0}^{n+2} = u^{n+1}$ and $p_{0}^{n+2} = p^{n+1}$. For sufficiently small~$\omega$, one can easily show that this converges to the BDF-$2$ solution as $K\to\infty$. For $K=1$, on the other hand, this only gives a first-order scheme. This can be expected, since the adaptation in the first equation may be interpreted as a discretization of a first-order delay approximation. 

For increasing~$K$, one can observe second-order convergence numerically for larger values of~$\tau$, which then reduces to first order for smaller step sizes. Hence, we expect second-order convergence only in the regime where $K$ scales with~$\tau$. 
% K ~ 1/log(tau) ??
We, however, are interested in an iterative scheme, which is second-order convergent for fixed $K$. For this, we consider a reformulation of~\eqref{eq:BDFimplicit} in terms of increments. 
%
%
%=============================================================================
\subsection{A matrix splitting for the increments} \label{sec:splitting-relaxation}
We rewrite~\eqref{eq:BDFsemiExpl} as 
\begin{equation*}
	\begin{bmatrix} \Amat & -\Dmat^T \\ 3 \Dmat & 3 \Ctau \end{bmatrix}
	\begin{bmatrix} \delta_{u}^{n+2} \\ \delta_{p}^{n+2} \end{bmatrix}
	=
	\begin{bmatrix} f^{n+2} - \Amat u^{n+1} + \Dmat^T p^{n+1} \\ 2\tau g^{n+2} - 2\tau \Bmat p^{n+1} + \Dmat(u^{n+1} - u^n) + \Cmat\,(p^{n+1} - p^n) \end{bmatrix}.
\end{equation*}
where $\delta_u^{n+2} \coloneqq u^{n+2} - u^{n+1}$ and $\delta_p^{n+2} \coloneqq p^{n+2} - p^{n+1}$. The application of the matrix splitting~\eqref{eq:matrixSplitting} leads to the iterative scheme
\begin{align}
	&\begin{bmatrix} \Amat & 0 \\ 3 \Dmat & 3 \Ctau \end{bmatrix}
	\begin{bmatrix} \delta_{u,k+1}^{n+2} \\ \delta_{p,k+1}^{n+2} \end{bmatrix} \notag \\
	&\hspace{1.6cm}=
	\begin{bmatrix} 0 & \Dmat^T \\ 0 & 0 \end{bmatrix}
	\begin{bmatrix} \delta_{u,k}^{n+2} \\ \delta_{p,k}^{n+2} \end{bmatrix} 
	+
	\begin{bmatrix} f^{n+2} - \Amat u^{n+1} + \Dmat^T p^{n+1} \\ 2\tau g^{n+2} - 2\tau \Bmat p^{n+1} + \Dmat(u^{n+1} - u^n) + \Cmat\,(p^{n+1} - p^n) \end{bmatrix}, \label{eq:scheme-delta}
\end{align}
where we initialize $\delta_{u,0}^{n+2} = u^{n+1} - u^n$ and $\delta_{p,0}^{n+2} = p^{n+1} - p^n$, and the approximations are given by $u^{n+2} \coloneqq u^{n+1} + \delta_{u,K}^{n+2}$ and $p^{n+2} \coloneqq p^{n+1} + \delta_{p,K}^{n+2}$.

Following the standard theory of iterative splitting schemes, this converges to the implicit solution for $K \to \infty$ if the spectral radius
\[
	\rho\left(
	\begin{bmatrix} \Amat & 0 \\ 3 \Dmat & 3 \Ctau \end{bmatrix}^{-1}
	\begin{bmatrix} 0 & \Dmat^T \\ 0 & 0 \end{bmatrix}
	\right)
	\le
	\rho(\Ctauinv\Dmat\Amat^{-1}\Dmat^T)
	\le
	\omega
\]
is less than $1$. For stronger coupling we can employ relaxation, for example for $\gamma > 0$ initialize
%
%\begin{subequations}
\begin{equation*}
	\delta_{u,0}^{n+2} 
	= u^{n+1} - u^n, \qquad 
	\delta_{p,0}^{n+2} 
	= p^{n+1} - p^n,
\end{equation*}
then iterate for $0 \le k \le K-1$
\begin{align*}
	&\begin{bmatrix} \Amat & 0 \\ 3 \Dmat & 3 \Ctau \end{bmatrix}
	\begin{bmatrix} \hat{\delta}_{u,k+1}^{n+2} \\ \hat{\delta}_{p,k+1}^{n+2} \end{bmatrix} \\
	&\hspace{1.2cm}=
	\begin{bmatrix} 0 & \Dmat^T \\ 0 & 0 \end{bmatrix}
	\begin{bmatrix} \delta_{u,k}^{n+2} \\ \delta_{p,k}^{n+2} \end{bmatrix} 
	+
	\begin{bmatrix} f^{n+2} - \Amat u^{n+1} + \Dmat^T p^{n+1} \\ 2\tau g^{n+2} - 2\tau \Bmat p^{n+1} + \Dmat(u^{n+1} - u^n) + \Cmat\,(p^{n+1} - p^n) \end{bmatrix},
\end{align*}
where except in the last step we dampen the pressure using
\begin{equation*}
	\delta_{u,k+1}^{n+2} = \hat{\delta}_{u,k+1}^{n+2}, \qquad \delta_{p,k+1}^{n+2} = \gamma \hat{\delta}_{p,k+1}^{n+2} + (1 - \gamma) \delta_{p,k}^{n+2},
\end{equation*}
and the approximation is given by
\begin{equation*}
	u^{n+2} \coloneqq u^{n+1} + \hat{\delta}_{u,K}^{n+2}, \qquad p^{n+2} \coloneqq p^{n+1} + \hat{\delta}_{p,K}^{n+2}.
\end{equation*}
%
%\end{subequations}
It can be seen that given only the coupling parameter the optimal choice for $\gamma$ is $\gamma \coloneqq 2/(2+\omega)$. For this choice we indeed get a convergence with the factor $\omega/(2+\omega) < 1$ in other words the scheme converges to the actual solution for $K \to \infty$.
%
%
%=============================================================================
\subsection{Novel iterative second-order scheme}
The scheme from the previous subsection can be written equivalently without the $\delta$ variables. Assume consistent initial data $p^0$, $u^0$ as well as $p^1$, $u^1$. For $n \ge 0$, we set the initial values of the inner iteration as  
\begin{subequations} 
\label{eq:scheme}
\begin{equation}
\label{eq:scheme:init}
	u_0^{n+2} 
	\coloneqq 2 u^{n+1} - u^n, \qquad
	p_0^{n+2} 
	\coloneqq 2 p^{n+1} - p^n.
\end{equation}
Then, for $0 \le k \le K-1$ solve the decoupled system 
\begin{equation}
\label{eq:scheme:hat}
	\begin{bmatrix} \Amat & 0 \\ 3 \Dmat & 3 \Ctau \end{bmatrix}
	\begin{bmatrix} \hat{u}_{k+1}^{n+2} \\ \hat{p}_{k+1}^{n+2} \end{bmatrix}
	=
	\begin{bmatrix} 0 & \Dmat^T \\ 0 & 0 \end{bmatrix}
	\begin{bmatrix} u_k^{n+2} \\ p_k^{n+2} \end{bmatrix}
	+
	\begin{bmatrix} f^{n+2} \\ 2\tau g^{n+2} + \Dmat(4 u^{n+1} - u^n) + \Cmat\,(4 p^{n+1} - p^n) \end{bmatrix}
\end{equation}
and, except for the last step, i.e., for $0 \le k \le K-2$, damp the pressure variable using
\begin{equation}
\label{eq:scheme:relaxation}
	u_{k+1}^{n+2} 
	\coloneqq \hat{u}_{k+1}^{n+2}, \qquad 
	p_{k+1}^{n+2} 
	\coloneqq \gamma\, \hat{p}_{k+1}^{n+2} + (1 - \gamma)\, p_k^{n+2} \qquad
	\text{with }\gamma = \tfrac{2}{2+\omega}.
\end{equation}
The final approximation is given by
\begin{equation}
\label{eq:scheme:final}
	u^{n+2} 
	\coloneqq \hat u^{n+2}_K, \qquad 
	p^{n+2} 
	\coloneqq \hat p^{n+1}_K.
\end{equation}
\end{subequations}
%
%Note that this differs from the first-order scheme presented in~\cite{AltD24} by the use of the BDF-$2$ scheme in the flow equation and the choice of the initial approximation $p_0^{n+2}$.
%
\begin{remark}
For $K=1$, this schemes reduces to the semi-explicit scheme of second order~\eqref{eq:BDFsemiExpl}.
\end{remark}
%
%
%=============================================================================
%=========  Convergence Analysis
%=============================================================================
\section{Convergence Analysis}\label{sec:convergence}
For the convergence proof of the newly introduced scheme~\eqref{eq:scheme}, we assume given (and exact) initial data $p^0$, $u^0$ as well as $p^1$, $u^1$, since we consider a $2$-step scheme. For both pairs, we assume consistency, i.e., 
\begin{align}
\label{eq:initConsistency}
	\Amat u^0 - \Dmat^T p^0 = \f^0, \qquad 
	\Amat u^1 - \Dmat^T p^1 = \f^1.
\end{align}
Moreover, we assume the existence of a positive constant~$C_\text{init}>0$ such that 
\begin{align}
\label{eq:initStability}
	\big\| p^1 - p(\tau) \big\|_\Ctau
%	= \|e_\text{init}\|_\Ctau
	\le C_\text{init} \tau^2. 
\end{align} 
Note that both assumptions can be easily matched, e.g., by one step of the implicit Euler scheme, which is locally of order two. 
\begin{proposition}
\label{prop:convergence-splitting}
Consider Assumption~\ref{ass:matrices} and let the right-hand sides of the semi-discrete problem~\eqref{eq:semiDiscretePoro} satisfy the smoothness conditions~$f \in C^3([0,T], \R^{n_u})$ and~$g \in C^2([0,T], \R^{n_p})$. Further assume initial data satisfying~\eqref{eq:initConsistency} and \eqref{eq:initStability} as well as~$K$ large enough such that 
\begin{equation}
\label{eq:iterationBound}
	3\,\omega^K < (2+\omega)^{K-1}.
\end{equation}
Then the iterates of the scheme~\eqref{eq:scheme} are stable in the sense that, for $n\ge 0$, 
\begin{subequations}
\begin{align}
	\big\| p^{n+2} - p_{K-1}^{n+2} \big\|_\Ctau 
	&\le C_1\,\tau^2, \label{eq:convergence-splitting:stability1} \\
	\sum_{k=1}^n \big\| (p^{k+2} - p_{K-1}^{k+2}) - (p^{k+1} - p_{K-1}^{k+1}) \big\|_\Ctau^2 
	&\le C_2\, \tau^4, \label{eq:convergence-splitting:stability2}
\end{align}
\end{subequations}
where $C_1$, $C_2$ are positive constants that only depend on the $C^3([0,T], \R^{n_u})$-norm of $f$, the $C^2([0,T], \R^{n_p})$-norm of~$g$, the time horizon~$T$, and the coupling parameter $\omega$. 
\end{proposition}
\begin{remark}
For $K=1$, where the scheme reduces to the semi-explicit scheme of second order, condition~\eqref{eq:iterationBound} reduces to $\omega < \frac13$, which is exactly the stability condition (numerically) observed in~\cite{AltMU24}.
\end{remark}
\begin{proof}[Proof of Proposition~\ref{prop:convergence-splitting}]
Consider the differences appearing in the first step of the inner iteration, namely 
\begin{align*} 
	\zeta_u^{n+2} 
	&\coloneqq \hat{u}^{n+2}_1 - u^{n+2}_0
	= \hat{u}^{n+2}_1 - 2 u^{n+1} + u^n, \\
	\zeta_p^{n+2} 
	&\coloneqq \hat{p}^{n+2}_1 - p^{n+2}_0
	= \hat{p}^{n+2}_1 - 2 p^{n+1} + p^n 
\end{align*}
for $n \ge 0$. These quantities are useful, since they satisfy -- as we will show in this proof -- the linear dependence
\begin{equation} 
\label{eq:convergence-splitting:linear}
	\hat{p}_{k+1}^{n+2} - p^{n+2}_k 
	= \Smat^k\zeta_p^{n+2}
\end{equation}
with the matrix~$\Smat$ from Lemma~\ref{lem:matrices} for $n\ge 0$ and $0 \le k \le K-1$. To see this, observe that by~\eqref{eq:scheme:hat} we get 
\[
	\hat{p}_k^{n+2} 
	= \Tmat p_{k-1}^{n+2} + r^{n+2}, 
\]
where
\[ 
	r^{n+2} 
	= \tfrac13\, \Ctauinv \big(2\tau g^{n+2} + \Dmat (4u^{n+1} - u^n) + \Cmat (4p^{n+1} - p^n) - 3\Dmat\Amat^{-1} f^{n+2} \big). 
\]
Combined with~\eqref{eq:scheme:relaxation}, this shows for all $1 \le k \le K-1$,  
\begin{equation} 
\label{eq:convergence-splitting:linear1}
	p_k^{n+2} 
	= \Smat p_{k-1}^{n+2} + \gamma\, r^{n+2}.
\end{equation}
On the other hand, we observe that 
\begin{align}
%\begin{split} 
	\hat{p}_{k+1}^{n+2}
	&= \Tmat\, \big(\gamma\, \hat{p}_k^{n+2} + (1 - \gamma)\, p_{k-1}^{n+2} \big) + r^{n+2} \notag \\
	&= \gamma\, \Tmat\hat{p}_k^{n+2} + (1 - \gamma) \big(\Tmat p_{k-1}^{n+2} + r^{n+2} \big) + \gamma\, r^{n+2} 
	= \Smat\hat{p}_k^{n+2} + \gamma\, r^{n+2}. \label{eq:convergence-splitting:linear2}
%\end{split}
\end{align}
Taking the difference of \eqref{eq:convergence-splitting:linear2} and~\eqref{eq:convergence-splitting:linear1}, we get 
\[
	\hat{p}_{k+1}^{n+2} - p_k^{n+2}
	= \Smat\big( \hat{p}_k^{n+2} - p_{k-1}^{n+2} \big).
\]
A repeated application of this result yields 
\[
	\hat{p}^{n+2}_{k+1} - p_k^{n+2}
	= \Smat\,\big( \hat{p}_k^{n+2} - p_{k-1}^{n+2} \big)
%	= \Smat^2 ( \hat{p}_{k-1}^{n+2} - p_{k-2}^{n+2} )
	= \dots 
	= \Smat^k \big( \hat{p}_{1}^{n+2} - p_{0}^{n+2} \big)
	= \Smat^k \zeta_p^{n+2}, 
\]
which equals the stated relation~\eqref{eq:convergence-splitting:linear}. In particular, we see that the quantity we are interested in can be described in terms of~$\zeta_p^{n+2}$, namely~$p^{n+2} - p_{K-1}^{n+2} = \Smat^{K-1}\zeta_p^{n+2}$. With~\eqref{eq:scheme:init} we further observe that 
\begin{align}
	p^{n+2} - 2p^{n+1} + p^n
	&= p^{n+2} - p^{n+2}_0 \notag \\
	&= p^{n+2} - p_{K-1}^{n+2} + \sum\nolimits_{k=0}^{K-2} \big( p_{k+1}^{n+2} - p_k^{n+2} \big) \notag \\
	&= \Smat^{K-1}\zeta_p^{n+2} + \sum\nolimits_{k=0}^{K-2} \big( \gamma\, \hat{p}_{k+1}^{n+2} + (1 - \gamma)\, p_k^{n+2} - p_k^{n+2} \big) \notag \\
	&= \Smat^{K-1}\zeta_p^{n+2} + \gamma\, \sum\nolimits_{k=0}^{K-2} \big( \hat{p}_{k+1}^{n+2} - p_k^{n+2} \big) \notag \\
	&= \Big(\Smat^{K-1} + \gamma\, \sum\nolimits_{k=0}^{K-2} \Smat^k\Big)\, \zeta_p^{n+2}
	= \tilde \Smat\, \zeta_p^{n+2}. \label{eq:convergence-splitting:zetap1}
\end{align}
To estimate the propagation of the sequence $\zeta_p^{n+2}$, we need to study the relationship between $\zeta_p^{n+2}$ and previous $\zeta_p^\ell$ for $\ell \le n+1$. For this, we first apply the second equation of~\eqref{eq:scheme:hat}. For $n+2$ and $k = 0$ it reads 
\begin{subequations}
\begin{equation} \label{eq:convergence-splitting:weighted:1}
	\Dmat(3\hat{u}_1^{n+2} - 4u^{n+1} + u^n) + \Cmat\,(3\hat{p}_1^{n+2} - 4p^{n+1} + p^n) + 2\tau\Bmat \hat{p}_1^{n+2} 
	= 2\tau g^{n+2}.
\end{equation}
We further need to consider the second equation of~\eqref{eq:scheme:hat} for the previous time step~$n+1$ and both, $k = 0$ and $k = K-1$. This yields 
\begin{align}
	\Dmat(3\hat{u}_1^{n+1} - 4u^n + u^{n-1}) + \Cmat\,(3\hat{p}_1^{n+1} - 4p^n + p^{n-1}) + 2\tau\Bmat \hat{p}_1^{n+1} 
	&= 2\tau g^{n+1}, \label{eq:convergence-splitting:weighted:2} \\
	\Dmat(3u^{n+1} - 4u^n + u^{n-1}) + \Cmat\,(3p^{n+1} - 4p^n + p^{n-1}) + 2\tau\Bmat p^{n+1} 
	&= 2\tau g^{n+1}. \label{eq:convergence-splitting:weighted:3}
\end{align}
\end{subequations}
Taking the weighted difference~$\eqref{eq:convergence-splitting:weighted:1} -  \frac13\cdot \eqref{eq:convergence-splitting:weighted:2} - \frac23\cdot \eqref{eq:convergence-splitting:weighted:3}$, we obtain the relation 
\begin{align} \label{eq:convergence-splitting:system1:a}
%	\Dmat(3\zeta_u^{n+2} &- \zeta_u^{n+1}) + \Cmat\,(3\zeta_p^{n+2} - \zeta_p^{n+1}) + \tfrac23\tau\Bmat (3\zeta_p^{n+2} - \zeta_p^{n+1}) \notag \\
	\Dmat(3\zeta_u^{n+2} &- \zeta_u^{n+1}) + \Ctau(3\zeta_p^{n+2} - \zeta_p^{n+1}) \notag \\
	&\qquad\quad=2\tau (g^{n+2} - g^{n+1}) - \tfrac23\tau\Bmat\big(4(p^{n+1} - p^n) -  (p^n - p^{n-1}) \big).
\end{align}
for $n \ge 1$. Since $p^{-1}$ is not defined, we need to consider the case $n=0$ separately. With~\eqref{eq:scheme:hat} we get 
\begin{align*}
%	3 \Dmat\zeta_u^2 + 3 \Cmat\zeta_p^2 + 2\tau\Bmat \zeta_p^2
	3 \Dmat\zeta_u^2 + 3 \Ctau\zeta_p^2 
	&= 3 \Dmat (\hat{u}^{2}_1 - 2 u^{1} + u^0) + 3 \Cmat_\tau (\hat{p}^{2}_1 - 2 p^{1} + p^0) \\
	% &= 2\tau g^{2} + \Dmat(4 u^{1} - u^0) + \Cmat\,(4 p^{1} - p^0)
	% + 3 \Dmat (- 2 u^{1} + u^0) + 3 \Cmat_\tau (- 2 p^{1} + p^0) \\
	&= 2\tau g^{2} - 2 \Dmat( u^{1} - u^0) -2 \Cmat\,( p^{1} - p^0)
	- 2 \tau \Bmat (2 p^{1} - p^0)
\end{align*}
as well as, using additionally the consistency conditions in~\eqref{eq:initConsistency}, 
\begin{align*}
	\Amat \zeta_u^2 
	= \Amat \big( \hat{u}^{2}_1 - 2 u^{1} + u^0 \big)
	= f^2 - 2f^1 + f^0.
\end{align*}
Inserting this in the previous equation, we obtain once again with~\eqref{eq:initConsistency}, 
\begin{align*}
%	3 \Cmat\zeta_p^2 + 2\tau\Bmat \zeta_p^2
	3 \Ctau\zeta_p^2 
	&= 2\tau g^{2} - 2\tau\Bmat(p^1 - p^0)\\
	&\qquad- 2 \Dmat\Amat^{-1}(\tfrac32 f^2 - 2 f^1 + \tfrac12 f^0) - 2\, (\Dmat\Amat^{-1}\Dmat + \Cmat)\,( p^{1} - p^0)
	- 2 \tau \Bmat p^{1}. \notag 
\end{align*}
Now, observe that $\tfrac32 f^2 - 2 f^1 + \tfrac12 f^0 = \tau \dot f(2\tau) + \calO(\tau^3)$, where the $\calO$ notion means that the difference of the two sides can be written with a vector, whose norm is bounded by a constant times $\tau^2$. In this particular estimate, the involved constant depends on the third derivative of~$f$. From assumption~\eqref{eq:initStability} we get $p^1 - p^0 = p(\tau) - p(0) + \calO(\tau^2) = \tau\dot p(\tau) + \calO(\tau^2)$. Together with the semi-discrete equation~\eqref{eq:semiDiscretePoro} at time $t=\tau$, this finally yields  
\begin{align}
%	3 \Cmat\zeta_p^2 + 2\tau\Bmat \zeta_p^2
	3 \Ctau\zeta_p^2 
	&= 2\tau g^{2} - 2\tau\Bmat(p^1 - p^0)  \notag \\
	&\qquad\qquad- 2\tau\, \big( \Dmat\Amat^{-1}\dot f(\tau) + (\Dmat\Amat^{-1}\Dmat^T + \Cmat)\dot p(\tau) + \Bmat p(\tau) \big) + \calO(\tau^2) \notag \\
	&= 2\tau\, (g^2 - g^1) - 2\tau\Bmat(p^1 - p^0) + \calO(\tau^2). \label{eq:convergence-splitting:zetap_init}
\end{align}
Going back to~\eqref{eq:convergence-splitting:system1:a}, we are now left with writing $3\zeta_u^{n+2} - \zeta_u^{n+1}$ in terms of~$\zeta_p$. For this, we similarly consider the first equation of~\eqref{eq:scheme:hat} for different~$n$, namely for $n+2$ with $k=0$ and for $n$, $n+1$ with $k=K-1$. In combination with~\eqref{eq:convergence-splitting:linear}, this leads to 
\begin{equation} \label{eq:convergence-splitting:system1:b}
	\Amat\zeta_u^{n+2} - \Dmat^T \Smat^{K-1} (2\zeta_p^{n+1} - \zeta_p^n) = \f^{n+2} - 2\f^{n+1} + \f^n
\end{equation}
for $n\ge2$. For the previous iterates, one can check that 
\[
	\Amat \zeta_u^2
%	= \Amat(\hat u^2_1 - 2 u^1 + u^0) 
	= \f^2 - 2\f^1 + f^0, \qquad
	\Amat \zeta_u^3 - \Dmat^T \Smat^{K-1} (2 \zeta_p^2) 
%	= \Amat(\hat u^3_1 - 2 u^2 + u^1) - \Dmat^T(2 p^2 - 2 p^2_{K-1}) 
	= \f^3 - 2\f^2 + f^1.
\]
Hence, setting $\zeta_p^0 = \zeta_p^1 = 0$, equation~\eqref{eq:convergence-splitting:system1:b} holds for all $n \ge 0$. 
Within the right-hand sides of \eqref{eq:convergence-splitting:system1:a} and \eqref{eq:convergence-splitting:system1:b}, we now use Taylor approximations. This leads to 
\[
	\f^{n+2} - 2\f^{n+1} + \f^n 
	= \calO(\tau^2), %\tau^2\ddot{\f}(\xi_f^{n+2}), 
	\qquad
	g^{n+2} - g^{n+1} 
	= \calO(\tau), %\tau\dot{g}(\xi_g^{n+2})
\]
where the hidden constants depend on the second derivative of $f$ and the first derivative of $g$, respectively. This turns~\eqref{eq:convergence-splitting:system1:b} into 
\[ 
	\zeta_u^{n+2} = \Amat^{-1}\Dmat^T\Smat^{K-1}\big(2\zeta_p^{n+1} - \zeta_p^n\big) + \calO(\tau^2) % \tau^2\Amat^{-1}\ddot{f}(\xi_f^{n+2}).
\]
Inserting this into equation~\eqref{eq:convergence-splitting:system1:a} yields, for $n\ge1$, 
\begin{align*} 
	\zeta_p^{n+2} 
	%&= \Tmat\Smat^{K-1}(2\zeta_p^{n+1} - \zeta_p^n)
	%+ \tfrac13\Ctauinv(\Dmat\zeta_u^{n+1} + \Ctau\zeta_p^{n+1}) \\
	%&\qquad- \tfrac29\tau\, \Ctauinv \Bmat \big[ 4 (p^{n+1} - p^n) - (p^n - p^{n-1}) \big]
	%- \tau^2\, \Ctauinv \big[ \Dmat\Amat^{-1} \ddot{f}(\xi_f^{n+2}) - \tfrac23\dot{g}(\xi_g^{n+2}) \big] \\
	%
%	&= \Tmat\Smat^{K-1}(2\zeta_p^{n+1} - \zeta_p^n)
%	-\tfrac13 \Tmat\Smat^{K-1}(2 \zeta_p^n - \zeta_p^{n-1})
%	+ \tfrac13\tau^2\Ctauinv\Dmat\Amat^{-1}\ddot{f}(\xi_f^{n+1})
%	+ \tfrac13\zeta_p^{n+1} \\
%	&\qquad- \tfrac29\tau\, \Ctauinv \Bmat \big[ 4(p^{n+1} - p^n) - (p^n - p^{n-1}) \big]
%	- \tau^2\, \Ctauinv \big[ \Dmat\Amat^{-1} \ddot{f}(\xi_f^{n+2}) - \tfrac23\dot{g}(\xi_g^{n+2}) \big] \\
	%
%	&= \Tmat\Smat^{K-1}(2\zeta_p^{n+1} - \tfrac53\zeta_p^n + \tfrac13\zeta_p^{n-1})
%	+ \tfrac13\zeta_p^{n+1} 
%	+ \tfrac23 \tau^2\, \Ctauinv \dot{g}(\xi_g^{n+2})  \\
%	&\quad- \tfrac29\tau\, \Ctauinv \Bmat \big[ 4 (p^{n+1} - p^n) - (p^n - p^{n-1}) \big] 
%	- \tau^2\, \Ctauinv \Dmat\Amat^{-1} \big[ \ddot{f}(\xi_f^{n+2}) - \tfrac13 \ddot{f}(\xi_f^{n+1}) \big]
	%
	&= \Tmat\Smat^{K-1}(2\zeta_p^{n+1} - \tfrac53\zeta_p^n + \tfrac13\zeta_p^{n-1})
	+ \tfrac13\zeta_p^{n+1}  \\
	&\qquad\qquad- \tfrac29\tau\, \Ctauinv \Bmat \big[ 4 (p^{n+1} - p^n) - (p^n - p^{n-1}) \big] + \calO(\tau^2).
\end{align*}
For $n=0$, we have already computed that 
\[
	\zeta_p^{2} 
%	= \tfrac23\, \tau\, \Ctauinv \big( \tau\, \dot{g}(\xi_g^{2}) - \Bmat (p^1 - p^0) \big) + \calO(\tau^2)
	= - \tfrac23 \tau\, \Ctauinv \Bmat (p^1 - p^0) + \calO(\tau^2).
\]
Either way, the terms including the right-hand sides (and the initial error coming from~\eqref{eq:initStability} for $n=0$) are uniformly bounded in the $\Ctau$-norm in terms of~$\Cfg \tau^2$ with 
\begin{align*}
%	&\big\| \Ctauinv \big[ \Dmat\Amat^{-1} (\ddot{f}(\xi_f^{n+2}) - \tfrac13 \ddot{f}(\xi_f^{n+1}) ) - \tfrac23\dot{g}(\xi_g^{n+2}) \big] \big\|_{\Ctau} + C_\text{init} \\
%	&\hspace{2.3cm}\le 
	\Cfg = \Cfg\big( \omega, \|f\|_{C^2([0,T];\R^{n_u})}, \|g\|_{C^1([0,T];\R^{n_p})}, C_\text{init} \big).
\end{align*}
For the terms preceded by $\Ctauinv\Bmat$, a uniform bound is not so obvious. Instead, we show the boundedness of these terms and~$\zeta_p^{n+2}$ simultaneously. For this, we introduce the block matrix
\[
	\calS 
	\coloneqq \begin{bmatrix}
	2\Tmat\Smat^{K-1} + \tfrac13 \mathbf{I} & -\tfrac53\Tmat\Smat^{K-1} & \tfrac13\Tmat\Smat^{K-1} \\
	\mathbf{I} & 0 & 0 \\
	0 & \mathbf{I} & 0
	\end{bmatrix}
	\in \R^{3n_p, 3n_p}.
\]
With $\tilde \zeta_p^{n+2} \coloneqq \big[(\zeta_p^{n+2})^T, (\zeta_p^{n+1})^T, (\zeta_p^n)^T \big]^T$, the above recursion formula (for $n\ge1$) reads 
\begin{align*}
	\tilde \zeta_p^{n+2}
%	= \calS \tilde \zeta_p^{n+1}
%	- \begin{bmatrix} \Ctauinv \\ 0 \\ 0 \end{bmatrix} \Bigg(\begin{array}{l}
%		\tfrac29\tau\, \Bmat \big[ 4 (p^{n+1} - p^n) - (p^n - p^{n-1}) \big] - \tfrac23\tau^2\, \dot{g}(\xi_g^{n+2}) \\
%		\qquad\quad+ \tau^2\, \Dmat\Amat^{-1} \big(\ddot{f}(\xi_f^{n+2}) - \tfrac13 \ddot{f}(\xi_f^{n+1})\big)  
%	\end{array}\Bigg)
	= \calS \tilde \zeta_p^{n+1}
	- \begin{bmatrix} \Ctauinv \\ 0 \\ 0 \end{bmatrix} \Big( \tfrac29\tau\, \Bmat \big[ 4 (p^{n+1} - p^n) - (p^n - p^{n-1}) \big] + \calO(\tau^2) \Big).
\end{align*}
In the following, we use the notion~$\|\bullet\|_{\Ctau}$ also for vectors in $\R^{3n_p}$, meaning the norm weighted by the block diagonal matrix~$\Diag(\Ctau,\Ctau,\Ctau)$. Together with the norm estimates from Lemma~\ref{lem:matrices}, we obtain 
\begin{align*}
	\|\zeta_p^{n+2}\|_\Ctau
	&\le \|\tilde \zeta_p^{n+2}\|_\Ctau \notag \\
	&\le \|\calS\|_\Ctau\, \|\tilde \zeta_p^{n+1}\|_\Ctau
	+ \tau\, C_p\, \big(4\, \|p^{n+1} - p^n\|_\Ctau + \|p^n - p^{n-1}\|_\Ctau \big) + \Cfg\tau^2 \\
	&\le \|\calS\|_\Ctau\, \|\tilde \zeta_p^{n+1}\|_\Ctau
	+ 5\tau\, C_p\, \max_{\ell\le n} \|p^{\ell+1} - p^\ell\|_\Ctau + \Cfg\tau^2
\end{align*}
for $n\ge1$ with $C_p \coloneqq 2C_\Bmat/(9c_\Cmat)$ as well as 
\begin{equation} \label{eq:convergence-splitting:zeta_initial}
	\|\zeta_p^2\|_\Ctau 
	= \|\tilde \zeta_p^2\|_\Ctau \le 3\tau\, C_p\, \|p^1 - p^0\|_\Ctau + \Cfg \tau^2.
\end{equation}
Applying these estimates recursively, we conclude that 
\begin{align}
	\|\zeta_p^{n+2}\|_\Ctau
	&\le \sum_{k=0}^{n} \|\calS\|_\Ctau^k \Big( 5\tau\, C_p\, \max_{\ell\le n} \|p^{\ell+1} - p^\ell\|_\Ctau + \Cfg \tau^2 \Big) \notag\\
	&\le \frac{1}{1 - \|\calS\|_\Ctau}\, \Big(5\tau\, C_p\, \max_{\ell\le n} \|p^{\ell+1} - p^\ell\|_\Ctau + \Cfg \tau^2 \Big), \label{eq:convergence-splitting:zetap2}
\end{align}
where we bound the sum by the geometric series and apply~$\|\calS\|_\Ctau < 1$, which is shown in Lemma~\ref{lem:recursion-matrix} under the assumption~\eqref{eq:iterationBound}. As such, the $\zeta_p$ quantities are bounded in terms of the differences of previous time steps. Together with~\eqref{eq:convergence-splitting:zetap1}, this leads to
\[
	\|p^{n+2} - p^{n+1}\|_\Ctau = \|\tilde \Smat \zeta_p^{n+2} + p^{n+1} - p^n\|_\Ctau
	\le \big(1 + 5\tau\, \theta\, C_p\big) \max_{\ell\le n} \|p^{\ell+1} - p^\ell\|_\Ctau + \theta\Cfg \tau^2
\]
with~$\theta \coloneqq 1 / (1 - \|\calS\|_\Ctau)$. Using the discrete 
Gr\"onwall lemma, we finally get 
\[
	\max_{\ell \le n} \big\| p^{\ell+1} - p^\ell \big\|_\Ctau 
	\le e^{5 t^n \theta C_p} \big(\|p^1 - p^0\|_\Ctau + \tau\, t^n\, \theta \Cfg \big).
\]
Due to $p^0=p(0)$ and assumption~\eqref{eq:initStability}, the expression on the right is of order~$\tau$. Inserting this estimate into~\eqref{eq:convergence-splitting:zetap2}, we get a constant~$C_1$ such that  
\begin{align}
	\label{eq:finalEstimateZeta}
	\|\zeta_p^{n+2}\|_\Ctau
	= C_1\, \tau^2. 
\end{align}
In combination with~\eqref{eq:convergence-splitting:linear} for $k=K-1$, this shows the first claim \eqref{eq:convergence-splitting:stability1}.

% claim 2
For the second claim, we define the differences 
\[
	\phi_u^{n+2} \coloneqq \zeta_u^{n+2} - \zeta_u^{n+1},
	\qquad
	\phi_p^{n+2} \coloneqq \zeta_p^{n+2} - \zeta_p^{n+1}
\]
for $n \ge 1$. Due to $\zeta_p^0 = \zeta_p^1 = 0$, we further set $\phi_p^2 \coloneqq \zeta_p^2$ and $\phi_p^1 \coloneqq 0$. By the formula derived for $\zeta_u^3$ and~\eqref{eq:convergence-splitting:zeta_initial} we get 
\[
	\zeta_u^3 
	= \Amat^{-1}(f^3 - 2 f^2 + f^1 + 2\Dmat^T \Smat^{K-1}\zeta_p^2) 
	= \calO(\tau^2). 
\]
Subtracting \eqref{eq:convergence-splitting:zetap_init} from \eqref{eq:convergence-splitting:system1:a} for $n = 1$ gives, together with~$\zeta_u^2, \zeta_u^3 = \calO(\tau^2)$,  
\begin{align*}
	\Ctau(3\phi_p^3 - \phi_p^2)   
	&= 2\tau (\g^3 - 2\g^2 + \g^1) - \tfrac83\tau \Bmat \tilde \Smat \zeta_p^2 - \Dmat(3\zeta_u^3 - \zeta_u^2) + \calO(\tau^2) \\
	&= -(\tfrac83\tau \Bmat \tilde \Smat + 3\Dmat\Amat^{-1}\Dmat^T) \zeta_p^2 + \calO(\tau^2). %\notag
\end{align*}
Finally, $\phi_p^2 = \zeta_p^2 = \calO(\tau^2)$ shows that also~$\phi_p^3$ is of order $2$. To bound the remaining~$\phi_p^n$, we consider the difference of equations~\eqref{eq:convergence-splitting:system1:a} and~\eqref{eq:convergence-splitting:system1:b} for consecutive~$n$. This yields 
\begin{align*}
	\Amat\phi_u^{n+2} - \Dmat^T \Smat^{K-1} (2\phi_p^{n+1} - \phi_p^n) 
	&= \f^{n+2} - 3\f^{n+1} + 3\f^n - \f^{n-1}, \\ %\label{eq:convergence-splitting:system2} \\
	\Dmat(3\phi_u^{n+2} - \phi_u^{n+1}) + \Ctau(3\phi_p^{n+2} - \phi_p^{n+1})   
	&= 2\tau (\g^{n+2} - 2\g^{n+1} + \g^n)  -\tfrac23\tau \Bmat \tilde \Smat (4\zeta_p^{n+1} - \zeta_p^n)
\end{align*}
for $n\ge2$. Note that we have used~\eqref{eq:convergence-splitting:zetap1} for the last term. The terms on the right-hand side of the second equation involving $\zeta_p^n$ are bounded due to~\eqref{eq:finalEstimateZeta}.
To obtain~\eqref{eq:convergence-splitting:stability2}, we proceed similarly as for the derivation of~\eqref{eq:convergence-splitting:zetap2} and observe that  
\begin{align}
	\|\phi_p^{n+2}\|_\Ctau
	\le \|\calS\|_\Ctau^n \|\phi_p^2\|_\Ctau + \tau^3 (\Cfg + 5C_1C_p) \sum_{\ell=0}^{n} \|\calS\|_\Ctau^\ell  
	\le C_1\, \tau^2\|\calS\|_\Ctau^n + \calO(\tau^3). \label{eq:convergence-splitting:phip}
\end{align}
%
%For the squared norm of $\phi_p^{n+2}$ we obtain
%\[
%	\|\phi_p^{n+2}\|_\Ctau^2 \le 2\tau^4\|\calS\|_\Ctau^{2n}C_1^2 + 2\tau^6\bigg(\frac{5\, C_1C_p}{1 - \|\calS\|_\Ctau}\bigg)^2. 
%\]
%
This directly leads to 
\begin{align*}
	\sum_{k = 0}^n \|\phi_p^{k+2}\|_\Ctau^2
	&\le \frac{2\, C_1^2}{1 - \|\calS\|_\Ctau^2} \tau^4 + t^n\calO(\tau^5).
\end{align*}
Together with~\eqref{eq:convergence-splitting:linear}, the second claim~\eqref{eq:convergence-splitting:stability2} follows.
\end{proof}
%

%\MD{ 
%%%  Motivation der neuen Annahme %%%
%	\begin{assumption}  \label{ass:initialStaticPressure}
%		We can reformulat the problem \eqref{eq:semiDiscretePoro}, by solving the first equation for $u$ and differentiating it:
%		%
%		\[ \Dmat\Amat^{-1} \dot f + (\Dmat\Amat^{-1}\Dmat^T + C)\dot p + \Bmat p = g \]
%		%
%		Considering this equation at time $t^2$ and using Taylor approximations obtain
%		%
%		\[ \Dmat\Amat^{-1} (3 f^2 - 4 f^1 + f^0) + 2 \tau (\Dmat\Amat^{-1}\Dmat^T + \Cmat)(\dot p(1) + \tau \ddot p(1)) + 2\tau \Bmat(p(1) + \tau \dot p(1)) = 2\tau g^2 + \calO(\tau^3) \]
%		%
%		We assume that $p^1$ and $p^0$ approximate $p$ in this formula sufficiently well, specifically
%		%
%		\[ \Dmat\Amat^{-1} (3 f^2 - 4 f^1 + f^0) + 2 (\Dmat\Amat^{-1}\Dmat^T + \Cmat)(p^1 - p^0) + 2\tau \Bmat(2p^1 - p^0) = 2\tau g^2 + \calO(\tau^3) \]
%		%
%		Which follows for example if $p^1 - p^0 = \tau \dot p(1) + \tau^2 \ddot p(1) + \calO(\tau^3)$.
%	\end{assumption}
%}

\begin{remark} \label{rem:additional-assumption}
Estimate~\eqref{eq:convergence-splitting:stability2} can be further improved under additional assumptions on the initial data. More precisely, we need to assume that the stability assumption~\eqref{eq:initStability} holds with one additional order and that the second derivative of the solution as well as of~$f$ are of order~$\tau$ at time $t = 0$, i.e., 
\[
	\big\| p^1 - p(\tau) \big\|_\Ctau
	\le C_\text{init} \tau^3 
	,\qquad
	\|\ddot p(0)\|_\Ctau \lesssim \tau
	\qquad \text{and} \qquad
	\|\ddot f(0)\|_\Ctau \lesssim \tau.
\]
With this, estimate~\eqref{eq:convergence-splitting:stability2} holds with $C_2 \tau^5$ on the right-hand side. 
%\[
%	\sum_{k=0}^n \big\| (p^{k+2} - p_{K-1}^{k+2}) - (p^{k+1} - p_{K-1}^{k+1}) \big\|_\Ctau^2 
%	\le C_2\, \tau^5, %\label{eq:convergence-splitting:stability2}
%\]
Note, however, that the two last conditions are not natural in the sense that it cannot be guaranteed by the choice of the initial data.  
\end{remark}
For the upcoming convergence analysis, we introduce an abbreviation of the BDF-$2$ scheme along with a useful identity. 
\begin{lemma} \label{lem:polarization}
With the notation	
\[
	\BDF_2\, \bullet^{n+2}
	\coloneqq \tfrac32 \bullet^{n+2} - 2 \bullet^{n+1} + \tfrac12\, \bullet^n
	= \tfrac32\, (\bullet^{n+2} - \bullet^{n+1}) - \tfrac12\,(\bullet^{n+1} - \bullet^n)
\]
we have for a symmetric and positive definite matrix $\Bmat$, 
\begin{align*}
	2\, (\BDF_2&\eta_p^{n+2})^T \Bmat\eta_p^{n+2}\\
	&= \BDF_2 \|\eta_p^{n+2}\|_\Bmat^2 + \|\eta_p^{n+2} - \eta_p^{n+1}\|_\Bmat^2 - \|\eta_p^{n+1} - \eta_p^n\|_\Bmat^2 + \tfrac12\|\eta_p^{n+2} - 2\eta_p^{n+1} + \eta_p^n\|.
\end{align*}
\end{lemma}
We are now in the position to prove second-order convergence of the newly introduced iteration scheme~\eqref{eq:scheme}. 
\begin{theorem}[Second-order convergence]
\label{thm:general-convergence} 
Consider the assumptions of Proposition~\ref{prop:convergence-splitting}, including $3\omega^K < (2 + \omega)^{K-1}$. Further, let the solution $(u, p)$ of~\eqref{eq:semiDiscretePoro} satisfy the smoothness conditions~$u \in C^3([0,T], \R^{n_u})$ and~$p \in C^3([0,T], \R^{n_p})$. Then scheme~\eqref{eq:scheme} converges with order $1.75$. More precisely, we have 
\begin{align*}
	\|u^n - u(t^n)\|_\Amat^2 + \|p^n - p(t^n)\|_\Cmat^2 
	\le C\, \tau^{7/2}
\end{align*}
with a positive constant~$C$ only depending on the solution, the right-hand sides, the time horizon, and the material parameters. 
Under the additional assumptions stated in Remark~\ref{rem:additional-assumption} the scheme converges with full order $2$.
\end{theorem}
\begin{proof}
For the difference between the (exact) semi-discrete solution and the outcome of scheme~\eqref{eq:scheme}, we introduce 
\[ 
	\eta_u^{n} \coloneqq u^{n} - u(t^{n}), \qquad
	\eta_p^{n} \coloneqq p^{n} - p(t^{n}).  
\]
The convergence proof is based on the following iterates 
\begin{align*}
	\nu_\Amat^{n+1} 
	&\coloneqq \tfrac32\, \|\eta_u^{n+1}\|_\Amat^2 - \tfrac12\, \|\eta_u^{n}\|_\Amat^2 + \|\eta_u^{n+1} - \eta_u^{n}\|_\Amat^2, \\
	\nu_\Bmat^{n+1} 
	&\coloneqq \tfrac32\, \|\eta_p^{n+1}\|_\Bmat^2 - \tfrac12\, \|\eta_p^{n}\|_\Bmat^2 + \|\eta_p^{n+1} - \eta_p^{n}\|_\Bmat^2, \\
	\nu_\Cmat^{n+1} 
	&\coloneqq \tfrac32\, \|\eta_p^{n+1}\|_\Cmat^2 - \tfrac12\, \|\eta_p^{n}\|_\Cmat^2 + \|\eta_p^{n+1} - \eta_p^{n}\|_\Cmat^2. 
\end{align*}
Which are sufficient for bounding the error, since we have by the triangle inequality
\begin{align}
	\nu_\Cmat^{n+1}
	&= \tfrac32\, \|\eta_p^{n+1}\|_\Cmat^2 - \tfrac12\, \|\eta_p^{n}\|_\Cmat^2 + \|\eta_p^{n+1} - \eta_p^{n}\|_\Cmat^2 \notag \\
	&\ge \tfrac32\, \|\eta_p^{n+1}\|_\Cmat^2 - \tfrac12\, \|\eta_p^{n+1}\|_\Cmat^2 + \tfrac12\|\eta_p^{n+1} - \eta_p^{n}\|_\Cmat^2 \notag \\
	&= \|\eta_p^{n+1}\|_\Cmat^2 + \tfrac12\, \|\eta_p^{n+1} - \eta_p^{n+1}\|_\Cmat^2 
	\ge \|\eta_p^{n+1}\|_\Cmat^2 \label{eq:general-convergence:nu-ineq-p}
\end{align}
and, analogously, $\|\eta_u^{n+1}\|_\Amat^2 \le \nu_\Amat^{n+1}$. 
Now let $n\ge 0$. Using the approximation property of BDF-$2$, namely 
\begin{align*}
	\tau\, \dot{u}(t^{n+2}) 
	&= \tfrac32 u(t^{n+2}) - 2 u(t^{n+1}) + \tfrac12 u(t^n) + \calO(\tau^3), \\
	\tau\, \dot{p}(t^{n+2}) 
	&= \tfrac32 p(t^{n+2}) - 2 p(t^{n+1}) + \tfrac12 p(t^n) + \calO(\tau^3),
\end{align*}
we obtain from equations~\eqref{eq:semiDiscretePoro} and~\eqref{eq:scheme:hat} for $k = K-1$ that 
\begin{subequations} \label{eq:general-convergence:system}
\begin{align}
	\Amat\eta_u^{n+2} - \Dmat^T\eta_p^{n+2} 
	&= \Dmat^T(p_{K-1}^{n+2} - p^{n+2}), \label{eq:general-convergence:system:a}\\
	\Dmat(\BDF_2 \eta_u^{n+2}) + \Cmat\,(\BDF_2 \eta_p^{n+2}) + \tau \Bmat\eta_p^{n+2} 
	&= b^{n+2}_{u,p} % \tfrac{1}{2} \tau^3 \big[\Dmat u^{(3)}(\xi_u^{n+2}) + \Cmat p^{(3)}(\xi_p^{n+2}) \big].
\end{align}
\end{subequations}
with $b^{n+2}_{u,p} = \calO(\tau^3)$. More precisely, the third-order term $b^{n+2}_{u,p}$ depends on the third derivatives of $u$ and $p$. With $\psi^{n+2} \coloneqq p_{K-1}^{n+2} - p^{n+2}$, equation~\eqref{eq:general-convergence:system:a} implies   
\[ 
	\Amat(\BDF_2 \eta_u^{n+2}) - \Dmat^T(\BDF_2 \eta_p^{n+2}) 
	= \Dmat^T (\BDF_2 \psi^{n+2}) 
\]
for $n \ge 2$. By setting $\psi^1 = \psi^0 = 0$, this formula also holds for $n = 0, 1$ due to the consistency assumption~\eqref{eq:initConsistency}. Now, we multiply the latter two equations from the left with $(\BDF_2 \eta_p^{n+2})^T$ and~$(\BDF_2 \eta_u^{n+2})^T$, respectively, and consider their sum. Using the inequality 
\begin{align}
	2\, (\BDF_2&\eta_p^{n+2})^T \Bmat\eta_p^{n+2} \notag \\
	&= \BDF_2 \|\eta_p^{n+2}\|_\Bmat^2 + \|\eta_p^{n+2} - \eta_p^{n+1}\|_\Bmat^2 - \|\eta_p^{n+1} - \eta_p^n\|_\Bmat^2 + \tfrac12\, \|\eta_p^{n+2} - 2\eta_p^{n+1} + \eta_p^n\| \notag \\
	&\ge \BDF_2 \|\eta_p^{n+2}\|_\Bmat^2 + \|\eta_p^{n+2} - \eta_p^{n+1}\|_\Bmat^2 - \|\eta_p^{n+1} - \eta_p^n\|_\Bmat^2 \notag \\
	&= \big(\tfrac32\, \|\eta_p^{n+2}\|_\Bmat^2 - \tfrac12\, \|\eta_p^{n+1}\|_\Bmat^2 + \|\eta_p^{n+2} - \eta_p^{n+1}\|_\Bmat^2\big) \notag \\
	&\qquad\quad- \big(\tfrac32\, \|\eta_p^{n+1}\|_\Bmat^2 - \tfrac12\, \|\eta_p^n\|_\Bmat^2 + \|\eta_p^{n+1} - \eta_p^n\|_\Bmat^2\big)
	= \nu_\Bmat^{n+2} - \nu_\Bmat^{n+1}, \label{eq:general-convergence:identity}
\end{align}
which stems from Lemma \ref{lem:polarization}, this yields 
\begin{align}
	&\|\BDF_2 \eta_u^{n+2}\|_\Amat^2 + \|\BDF_2 \eta_p^{n+2}\|_\Cmat^2 + \tfrac{1}{2} \tau\, \big(\nu_\Bmat^{n+2} - \nu_\Bmat^{n+1}\big) \notag \\
	&\qquad\quad\le \|\BDF_2 \eta_u^{n+2}\|_\Amat^2 + \|\BDF_2 \eta_p^{n+2}\|_\Cmat^2 + \tau\,(\BDF_2 \eta_p^{n+2})^T \Bmat\eta_p^{n+2} \notag \\
	&\qquad\quad= (\BDF_2 \eta_u^{n+2})^T\Dmat^T (\BDF_2 \psi^{n+2}) + (\BDF_2 \eta_p^{n+2})^T b_{u,p}.
	\label{eq:general-convergence:ineq-a}
\end{align}
Using the definition of the coupling strength~$\omega$ and Young's inequality, the first summand of the right-hand side of \eqref{eq:general-convergence:ineq-a} can be bounded by 
\begin{align*}
	(\BDF_2 \eta_u^{n+2})^T\Dmat^T (\BDF_2 \psi^{n+2})
	% \qquad\le \sqrt{\omega}\, \|\BDF_2 \psi^{n+2}\|_\Cmat \, \|\BDF_2 \eta_u^{n+2}\|_\Amat \\ 
	\le \tfrac\omega2\, \|\BDF_2 \psi^{n+2}\|_\Cmat^2 + \tfrac12\, \|\BDF_2 \eta_u^{n+2}\|_\Amat^2 .
\end{align*}
For the second summand of \eqref{eq:general-convergence:ineq-a} we get  
\begin{align*}
%	\tfrac{1}{2} \tau^3 (\BDF_2 \eta_p^{n+2})^T &\big[\Dmat u^{(3)}(\xi_u^{n+2}) + \Cmat p^{(3)}(\xi_p^{n+2}) \big] \\
%	&\quad\le \tfrac{1}{2} \tau^3\, \|\BDF_2 \eta_p^{n+2}\|_\Cmat
%	\big( \sqrt{\omega}\, \|u^{(3)}(\xi_u^{n+2})\|_\Amat + \| p^{(3)}(\xi_p^{n+2})\|_\Cmat  \big)  \\ 
%	&\quad\le \tfrac18 \tau^6 \big( \omega\, \|u^{(3)}\|^2_{L^\infty(\Amat)} + \|p^{(3)}\|_{L^\infty(\Cmat)}^2 \big) + \|\BDF_2 \eta_p^{n+2}\|_\Cmat^2. 
	%
	(\BDF_2 \eta_p^{n+2})^T b_{u,p}
	\le \|\BDF_2 \eta_p^{n+2}\|_\Cmat^2 + \| b_{u,p}\|^2
	= \|\BDF_2 \eta_p^{n+2}\|_\Cmat^2 + \calO(\tau^6).
\end{align*}
Absorbing~$\tfrac12\, \|\BDF_2 \eta_u^{n+2}\|_\Amat^2$ and $\|\BDF_2 \eta_p^{n+2}\|_\Cmat^2$ and multiplying both sides by $2$, we obtain 
\begin{align}
	\|\BDF_2 \eta_u^{n+2}\|_\Amat^2 + \tau\, \big(\nu_\Bmat^{n+2} - \nu_\Bmat^{n+1}\big)
	\le \omega\, \|\BDF_2 \psi^{n+2}\|_\Ctau^2 + \calO(\tau^6).
%	+ \tfrac14\, \tau^6 \big( \omega\, \|u^{(3)}\|^2_{L^\infty(\Amat)} + \|p^{(3)}\|_{L^\infty(\Cmat)}^2 \big). 
	\label{eq:general-convergence:single-step}
\end{align}
%
% \begin{align}
% \nu_\Bmat^{n+2}
% &= \tfrac32\|\eta_p^{n+2}\|_\Bmat^2 - \tfrac12 \|\eta_p^{n+1}\|_\Bmat^2 + \|\eta_p^{n+2} - \eta_p^{n+1}\|_\Bmat^2 \notag \\
% &\ge \tfrac32\|\eta_p^{n+2}\|_\Bmat^2 - \tfrac12 \|\eta_p^{n+2}\|_\Bmat^2 + \tfrac12\|\eta_p^{n+2} - \eta_p^{n+1}\|_\Bmat^2 \notag \\
% &= \|\eta_p^{n+2}\|_\Bmat^2 + \tfrac12\|\eta_p^{n+2} - \eta_p^{n+1}\|_\Bmat^2 \ge \|\eta_p^{n+2}\|_\Bmat^2 \label{eq:general-convergence:nu-ineq}
% \end{align}
%
We now again consider the system~\eqref{eq:general-convergence:system}. This time, however, using~$\BDF_2 \eta_u^{n+2}$ and~$\eta_p^{n+2}$ as test functions, respectively. The sum then reads 
\begin{align*}
	&(\BDF_2 \eta_u^{n+2})^T\Amat\eta_u^{n+2} +
	(\eta_p^{n+2})^T\Cmat\,(\BDF_2 \eta_p^{n+2}) + \tau\, \|\eta_p^{n+2}\|_\Bmat^2 \\
	&\qquad\qquad= (\BDF_2 \eta_u^{n+2})^T\Dmat^T\psi^{n+2} + (\eta_p^{n+2})^Tb^{n+2}_{u,p}.
	%+ \tfrac{1}{2}\, \tau^3 (\eta_p^{n+2})^T\big[\Dmat u^{(3)}(\xi_u^{n+2}) + \Cmat p^{(3)}(\xi_p^{n+2}) \big]. 
\end{align*}
Considering an inequality as~\eqref{eq:general-convergence:identity} but for $\nu_\Amat^{n+2}$ and $\nu_\Cmat^{n+2}$, this can be transformed to  
\begin{align*}
	\nu^{n+2}_\Amat - \nu^{n+1}_\Amat + \nu^{n+2}_\Cmat - \nu^{n+1}_\Cmat + 2\tau\, \|\eta_p^{n+2}\|_\Bmat^2 
	\le 2\, (\BDF_2 \eta_u^{n+2})^T\Dmat^T\psi^{n+2} + 2\, (\eta_p^{n+2})^Tb^{n+2}_{u,p}.
\end{align*}
Now, let $\beta$ be the coefficient appearing in the right-hand side of~\eqref{eq:convergence-splitting:stability2}, i.e., $\beta=4$ in the general case and $\beta=5$ in the special case considered in Remark~\ref{rem:additional-assumption}. Using once more Young's inequality, we bound the right-hand side by 
\begin{align*}
	\tau^{-\frac{\beta-3}{2}}\, \|\BDF_2 \eta_u^{n+2}\|_\Amat^2 
	+ \omega\, \tau^{\frac{\beta-3}{2}}\, \|\psi^{n+2}\|_\Cmat^2 
	+ 2\tau\, \|\eta_p^{n+2}\|_\Bmat^2 
	+ \calO(\tau^5).
\end{align*}
By \eqref{eq:general-convergence:single-step} this is bounded by  
\begin{align*}
	\omega\, \tau^{-\frac{\beta-3}{2}}\, \|\BDF_2 \psi^{n+2}\|_\Ctau^2
	- \tau^{\frac{5-\beta}{2}}\, \big(\nu_\Bmat^{n+2} - \nu_\Bmat^{n+1}\big)
	+ \omega\, \tau^{\frac{\beta-3}{2}}\, \|\psi^{n+2}\|_\Cmat^2 
	+ 2\tau\, \|\eta_p^{n+2}\|_\Bmat^2 
	+ \calO(\tau^5).
\end{align*}
Absorbing $2\tau\, \|\eta_p^{n+2}\|_\Bmat^2$ and applying~\eqref{eq:convergence-splitting:stability1}, i.e., $\| \psi^{n+2} \|_\Ctau \le C_1\,\tau^2$, the full inequality reads 
\begin{align*}
	&\nu^{n+2}_\Amat - \nu^{n+1}_\Amat + \nu^{n+2}_\Cmat - \nu^{n+1}_\Cmat 
	+ \tau^{\frac{5-\beta}{2}}\, \big(\nu_\Bmat^{n+2} - \nu_\Bmat^{n+1}\big) \\
	&\hspace{2.6cm}\le \omega\, \tau^{-\frac{\beta-3}{2}}\, \|\BDF_2 \psi^{n+2}\|_\Ctau^2
	+ \omega\, \tau^{\frac{\beta-3}{2}}\, \|\psi^{n+2}\|_\Cmat^2  
	+ \calO(\tau^5) \\
	&\hspace{2.6cm}\le \omega\, \tau^{-\frac{\beta-3}{2}}\, \|\BDF_2 \psi^{n+2}\|_\Ctau^2
	+ \omega\,C_1^2\, \tau^{\frac{\beta+5}{2}}
	+ \calO(\tau^5).
\end{align*}
Finally, we sum over $n$ and use $\BDF_2 \psi^{n+2} = \tfrac32\, (\psi^{n+2} - \psi^{n+1}) - \tfrac12\, (\psi^{n+1} - \psi^n)$, which allows the application of~\eqref{eq:convergence-splitting:stability2} or the improved estimate from Remark~\ref{rem:additional-assumption} as long as $n \ge 2$. For $n=0,1$, the BDF-terms are bounded due to $\psi^1 = \psi^0 = 0$ and~$\|\psi^2\|_\Ctau \le C_1\, \tau^2$. Moreover, applying \eqref{eq:general-convergence:nu-ineq-p} as well as $\nu_\Amat^1 = \calO(\tau^4)$ and $\nu_\Bmat^1 = \nu_\Cmat^1 = \calO(\tau^4)$, we get 
\begin{align*}
	\|\eta_u^n\|_\Amat^2 
%	+ \tau^{\frac{5-\beta}{2}}\, \|\eta_p^n\|_\Bmat^2 
	+ \|\eta_p^n\|_\Cmat^2 
	\le \nu_\Amat^{n} + \tau^{\frac{5-\beta}{2}} \nu_\Bmat^{n+2} + \nu_\Cmat^{n} 
	\le \omega\, \tau^{\frac{\beta+3}{2}} \big( 4 C_2 + 2T C_1^2 \big) + T\, \calO(\tau^4).
\end{align*}
As stated, $\beta = 4$ leads to a convergence rate of $1.75$, whereas for $\beta = 5$ we achieve the full rate of $2$. 
\end{proof}
\begin{remark}
The iteration bound~\eqref{eq:iterationBound} can be solved for $K$, leading to the requirement
\[ 
	K 
	> K(\omega)
	\coloneqq 1 + \frac{\log\omega + \log 3}{\log(\omega + 2) - \log(\omega)}. 
\]
Since this is well-defined for all positive~$\omega$, we see that stability and (second-order) convergence can be achieved for all choices of the coupling strength~$\omega$ with an appropriate number of inner iteration steps~$K$.
\end{remark}
%
%
%=============================================================================
%=========  Numerics
%=============================================================================
\section{Numerical Examples}\label{sec:numerics}
In this section, we verify our convergence results experimentally. We first use the model problem from \cite{AltMU21} and a simple poroelastic example to analyze the sharpness of our proven iteration bound \eqref{eq:iterationBound}. Following this, we apply the scheme to a more realistic three-dimensional problem modeling pressure changes in the brain. For this, we also compare the runtimes of different numerical approaches.
%
%
%=============================================================================
\subsection{Sharpness of the iteration bound}\label{sec:numerics:sharpness}
For a study of the iteration bound \eqref{eq:iterationBound}, we consider the model problem from~\cite{AltMU24}. More precisely, we consider a finite-dimensional setting with the bilinear forms 
\begin{subequations}
\label{eq:modelProblem}
\begin{align}
	a(u, v) = v^T\Amat u, \quad
	b(p, q) = q^T\Bmat p, \quad 
	c(p, q) = q^T\Cmat p, \quad
	d(v, p) = \sqrt{\tilde \omega}\, p^T\Dmat v
\end{align}
and matrices
\begin{align}
	\Amat = \tfrac{1}{2 - \sqrt{2}} 
	\begin{bmatrix} 2 & -1 & 0 \\ -1 & 2 & -1 \\ 0 & -1 & 2 \end{bmatrix}, \qquad
	\Bmat = 1, \qquad
	\Cmat = 1, \qquad
	\Dmat = \tfrac{1}{3}\, \big[\ 2\ \ 1\ \ 2\ \big]. %\begin{bmatrix} 2 & 1 & 2 \end{bmatrix}. 
\end{align}
\end{subequations}
The coupling parameter is then indeed given by $\tilde \omega$, since the matrices are normalized such that $c_\Amat = c_\Cmat = 1$ and $C_\Dmat = \sqrt{\tilde \omega}$. Additionally, we apply the right-hand sides
\[ 
	f(t) \equiv \big[\ 1\ \  1\ \  1\ \big]^T, \qquad 
	g(t) = \sin(t)
\]
and set as initial condition~$p^0=p(0)=1$. We then compute $u^0$ in a consistent manner and approximations $p^1$ and $u^1$ using one step of the implicit Euler scheme. Recall that this results in initial data, which satisfies the consistency assumption~\eqref{eq:initConsistency} as well as the stability assumption~\eqref{eq:initStability}.

We run the scheme for $2^{10}$ time steps, i.e., until time $T = 1$ for $K = 1, \dots, 5$ and compare the results to a reference solution obtained by the midpoint scheme. We observe that our prediction~\eqref{eq:iterationBound} is exceeded by a moderate factor, cf.~Figure~\ref{fig:sharpnessComparison}. Therein, we plot the theoretically needed number of inner iterations $K(\tilde \omega)$ as well as numerically observed bounds of $\tilde \omega$ for each $K$, for which the error seems to stay bounded. In particular, one can observe that even values of $K$ perform much better regarding the stability. Using $K = 2$ or $K = 4$, for example, seems better than using $K = 3$ or $K = 5$, respectively. 
\begin{figure}
	\begin{tikzpicture}
		\begin{axis}[
			width=0.75\textwidth,
			height=0.4\textwidth,
			axis lines = left,
			xlabel = coupling paramter~$\tilde \omega$,
			ylabel = \(K(\tilde \omega)\),
			xmin = 0, xmax=8.4,
			ymin = 0, ymax=10.9,
			grid = both,
		    grid style={line width=.1pt, draw=gray!10},
		    major grid style={line width=.2pt,draw=gray!50},
			minor x tick num=1,
			minor y tick num=1,
			legend pos = north west,
			legend cell align = {left},
		]
			\addplot [
				domain = 0.1:8,
				samples = 100,
				variable = \t,
				color = mycolor4,
				very thick,
			]
			({t}, {1 + ln(3*t)/(ln((t + 2)/t))});
%			\addplot [  % for even K
%				domain = 0.1:8,
%				samples = 100,
%				variable = \t,
%				color = mycolor4,
%				dashed,
%				very thick,
%			]
%			({t}, {1 + ln(t)/(ln((t + 2)/t))}); 
			\addplot [
			mark = pentagon*,
			mark size = 3,
			color = mycolor3,
			only marks,
			] coordinates
			{
				(0.32, 1)
				(2.8, 2)
				(2.5, 3)
				(7.3, 4)
				(5.6, 5)
			};
			\addplot [
			mark = o,
			thick,
			mark size = 3,
			color = mycolor5,
			only marks,
			] coordinates
			{
				(0.3, 1)
				(2.9, 2)
				(2.6, 3)
				(8.0, 4)
			};
			\legend{proven bound~\eqref{eq:iterationBound}, model problem \eqref{eq:modelProblem}, poro problem \eqref{eq:sharpness:poroProblem}}
		\end{axis}
	\end{tikzpicture}
	\caption{Comparison of the proven iteration bound~\eqref{eq:iterationBound} and numerically observed values for the model problem \eqref{eq:modelProblem} and the poroelastic example~\eqref{eq:sharpness:poroProblem} }
	\label{fig:sharpnessComparison}
\end{figure}
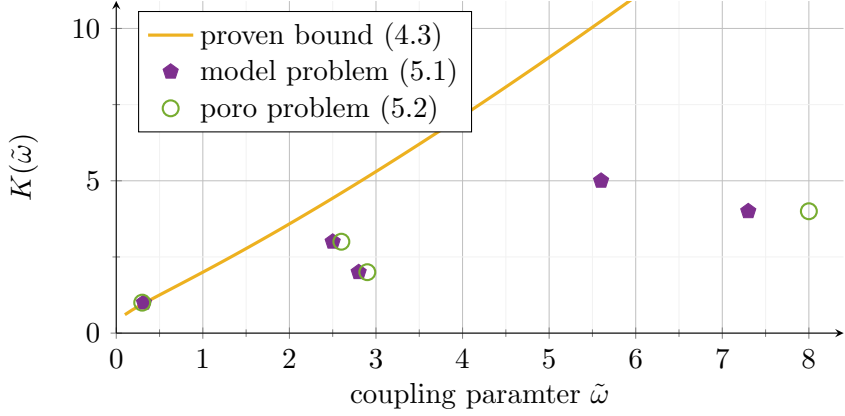

We also run a similar test using an actual poroelastic model problem. On a unit square domain we set
\begin{subequations}
\label{eq:sharpness:poroProblem}
\begin{equation}
M = \kappa = \nu = 1, \qquad 
\mu = \lambda = \tfrac12, \qquad 
\alpha = \sqrt{\omega}, \\
\end{equation}
\begin{equation}
f(x, y, t) = xy (1 - x) (1 - y), 
\qquad g(x, y, t) = \sin(t).
\end{equation}
\end{subequations}
Again, one can see that the coupling parameter is given by $\omega$. We choose $u^0$ and $p^0$ to be the solution to the static problem at $t = 0$, i.e., the solution to the equation
\begin{subequations} \label{eq:static}
\begin{align}
	\Amat u^0 - \Dmat^T p^0 
	&= f(0), \\
	\Bmat p^0 
	&= g(0).
\end{align}
\end{subequations}
As before, we use a single step of the implicit Euler method to find $u^1$ and $p^1$ such that assumptions~\eqref{eq:initConsistency} and~\eqref{eq:initStability} are satisfied. We observe a very similar behaviour regarding the iteration bound as for the model problem~\eqref{eq:modelProblem}; see Figure~\ref{fig:sharpnessComparison}.
%
%
%=============================================================================
\subsection{Simulation of 3D brain tissue} \label{sec:numerics:brain}
We now turn to a more practical example taken from medical literature \cite{JuCLT20}. In this paper, the poroelastic response of brain tissue as a result of increased pressure is simulated on a 2D slice of the brain. Here, this model is extended to a 3D mesh obtained from~\cite{CZK+98, Fang10}; see Figure~\ref{fig:mesh}. Due to the large number of degrees of freedom and the bad conditioning of the matrices, % in 3D
we expect notable advantages of the introduced scheme compared to fully implicit methods. 
\begin{figure}
\includegraphics[width=0.5\textwidth]{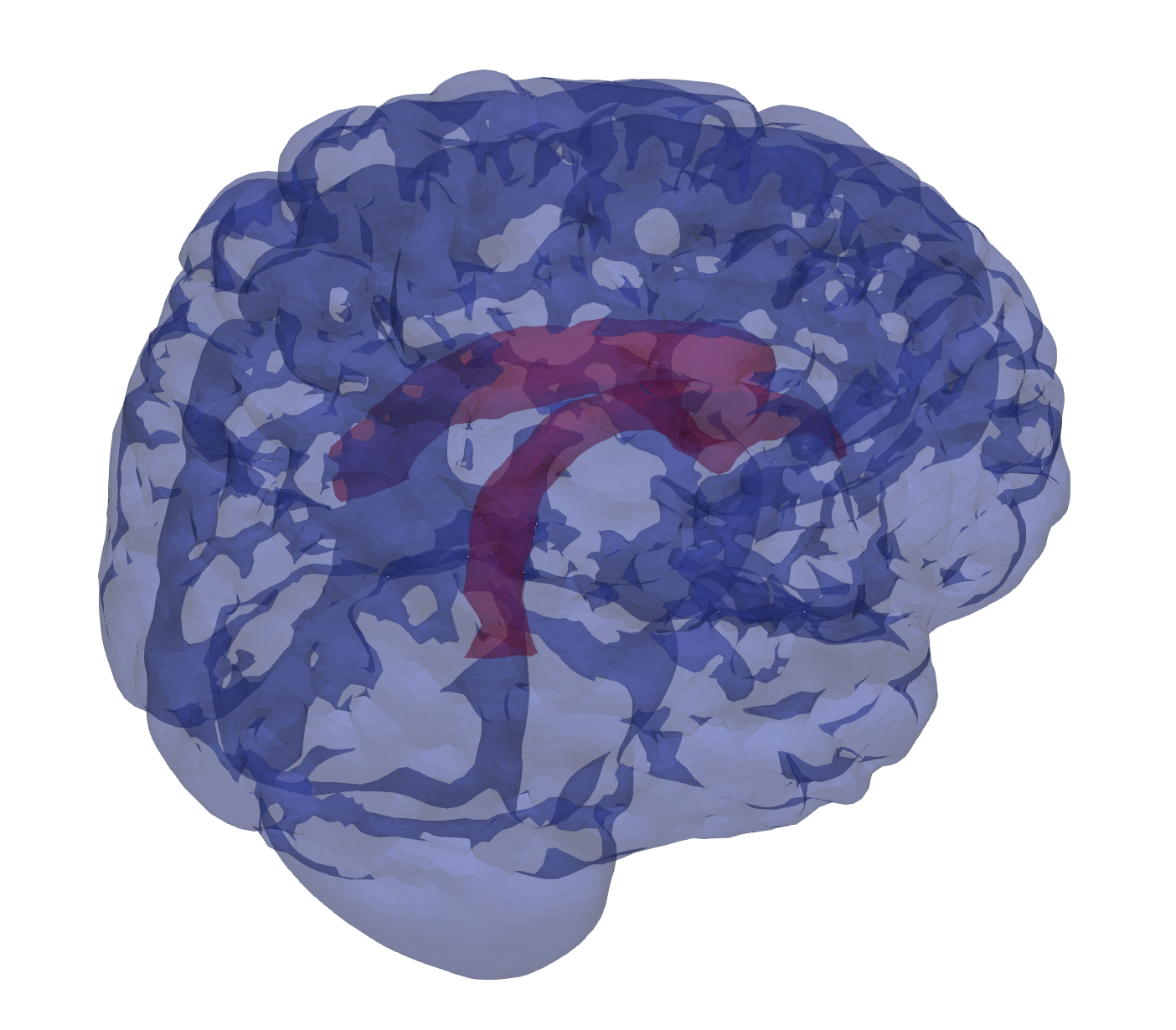}
\caption{ \label{fig:mesh}
	The 3D mesh used for the brain tissue simulation obtained from \cite{CZK+98, Fang10}. The boundary is divided into an interior part~$\Gamma_\text{int}$ (red) and an exterior part~$\Gamma_\text{ext}$ (blue).	}
\end{figure}

We use the same material constants as in~\cite{AltD24}, i.e., the parameters from~\cite[Tab.~10]{JuCLT20} with an adjustment of the Biot modulus~$M$, leading to a coupling strength of~$\omega\approx 2.8$. This is justified by other values used in medical literature, cf.~\cite{EliRT23,JuCLT20,PieRV23}. More precisely, we set 
\begin{align*}
	\lambda &= 7.8\times10^3 \newton/\meter^2, &
	\mu &= 3.3\times10^3 \newton/\meter^2, &
	\alpha &= 1, \\
	\kappa &= 1.3\times10^{-15} \meter^2, &
	\nu &= 8.9\times10^{-4} \newton\second/\meter^2, &
	M &= 2.2\times10^{4} \newton/\meter^2.
\end{align*} 
Additionally, we impose boundary conditions on the exterior and interior boundary of the mesh, which we will call $\Gamma_\text{ext}$ and $\Gamma_\text{int}$, respectively. These conditions are given by
\begin{subequations} 
\label{eq:brain:bc}
\begin{align}
	u &= 0 && \text{on } \Gamma_\text{ext}, \\
	\left(\tfrac{\kappa}{\nu}\, \nabla p\right) \cdot n 
	&= c_{\operatorname{SAS}} (p_{\operatorname{SAS}} - p) && \text{on } \Gamma_\text{ext}, \\
	(\sigma(u) - \alpha p) \cdot n 
	&= -p \cdot n && \text{on } \Gamma_\text{int}, \label{eq:brain:neumann-bc} \\
	p 
	&= 1100 \newton/\meter^2 && \text{on } \Gamma_\text{int}, \label{eq:brain:dirichlet-bc}
\end{align}
\end{subequations}
where $c_{\operatorname{SAS}} = 5.0\times10^{-10} \meter^3/(\newton\second)$ is the conductance, $p_{\operatorname{SAS}} = 1070 \newton/\meter^2$ the pressure outside the brain tissue wall, and $n$ the outer normal unit vector. Again we choose initial conditions as the solution of the static problem \eqref{eq:static} with $f \equiv 0$ and $g \equiv 0$. This corresponds to the neutral state of the brain. As right-hand side, we set 
\[
	g(x, t) = \begin{cases} 1.5\times10^{-4}\second^{-1} & x \in \Omega_d \\ 0 & x \notin \Omega_d \end{cases},
\]
where $\Omega_d \subseteq \Omega$ is a small region on the right side of the mesh, simulating increased pressure as a result of this part of the tissue being damaged. As before, we compute appropriate values $u^1$ and $p^1$ by the implicit Euler scheme. This time, however, using four smaller time steps in order to increase the accuracy of the data.

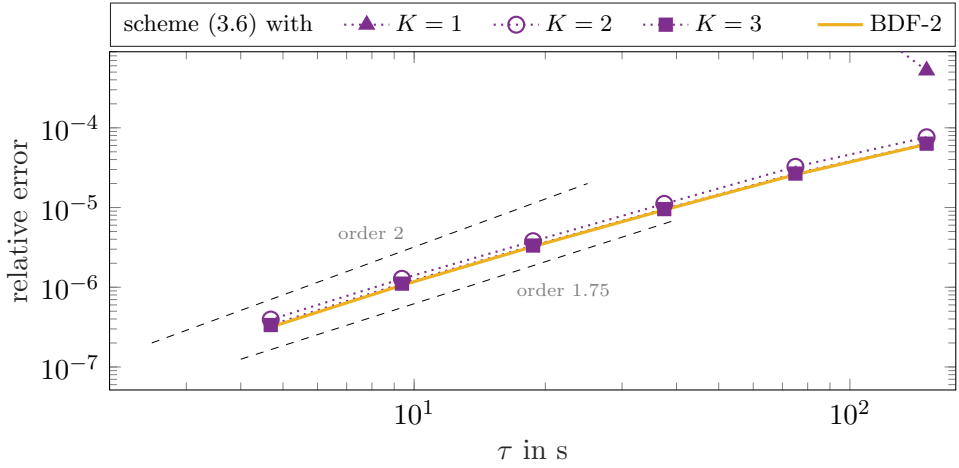
\begin{figure}
%%%%%%% omega = 2.8 %%%%%%%%
\begin{tikzpicture}
		\begin{axis}[%
		width=0.75\textwidth,
		height=0.3\textwidth,
		at={(0\textwidth,0\textwidth)},
		scale only axis,
		unbounded coords=jump,
		xlabel style={font=\color{white!15!black}},
		xlabel={$\tau$ in s},
		xmode=log,
		ymode=log,
		% yminorticks=true,
		ylabel style={font=\color{white!15!black}},
		ylabel={relative error},
		axis background/.style={fill=white},
		legend columns = 5,
		legend style={
			at={(1.0,1.08)},
			anchor=east,
			legend cell align=left,
			align=left,
			draw=white!15!black,
			font=\Small
		},
		xmin = 2,
		xmax = 175,
		ymax = 0.9e-3,
		% mark repeat=2,
		% mark phase=4,
		mark size=3.0,
		]
		\addlegendimage{empty legend}
		\addlegendentry{scheme~\eqref{eq:scheme} with\quad}

		\addplot [color = mycolor3, dotted, mark=triangle*, thick, every mark/.append style={solid}] table [col sep=semicolon, x index = {0}, y index = {1}] {experiments/output/convergence_3d.csv};
		\addlegendentry{$K = 1$\quad}
		\addplot [color = mycolor3, dotted, mark=o, thick, every mark/.append style={solid}] table [col sep=semicolon, x index = {0}, y index = {2}] {experiments/output/convergence_3d.csv};
		\addlegendentry{$K = 2$\quad}
		\addplot [color = mycolor3, dotted, mark=square*, thick, every mark/.append style={solid}, mark size=2.3] table [col sep=semicolon, x index = {0}, y index = {3}] {experiments/output/convergence_3d.csv};
		\addlegendentry{$K = 3$\qquad}
		
		\addlegendentry{BDF-$2$}
		
		\addplot [color = mycolor4, very thick, every mark/.append style={solid}] table [col sep=semicolon, x index = {0}, y index = {4}] {experiments/output/convergence_3d.csv};

		% order two
		\addplot [black, dashed, very thin] coordinates {
			(0.25e1, 0.2e-6) 
			(0.25e2, 0.2e-4) 
		};
		\node[anchor=south east, gray] at (axis cs:1e1, 0.3e-5) {\tiny order $2$};

		% order 1.75
		\addplot [black, dashed, very thin] coordinates {
			(0.4e1, 0.125e-6) 
			(0.4e2, 0.25*0.281e-4) 
		};
		\node[anchor=north, gray] at (axis cs:2.2e1, 1.5e-6) {\tiny order $1.75$};

%		% order 1
%		\addplot [black, dashed, very thin] coordinates {
%			(1e1, 1e-6) 
%			(1e2, 1e-5) 
%		};

		\end{axis}
		\end{tikzpicture}

		\caption{ \label{fig:brain-convergence}
			Relative error of the proposed and the implicit BDF-$2$ scheme as an $L^2$ average over the timespan $t \in [0, 10 \minute]$. The solution is compared to a reference computed using a midpoint scheme.
		}
\end{figure}

We compare our scheme to the fully implicit BDF-$2$ scheme~\eqref{eq:BDFimplicit} for different time step sizes, cf.~Figure~\ref{fig:brain-convergence}. The reference solution was computed with the midpoint scheme.  
% The error is computed as
% \[\frac{\big(\sum_{n = 0}^N \|u^n - u(t^n)\|_\Amat^2 + \|p^n - p(t^n)\|_\Cmat^2 \big)^{1/2}}{\big(\sum_{n = 0}^N \|u(t^n)\|_\Amat^2 + \|p(t^n)\|_\Cmat^2 \big)^{1/2}}. \]
One can observe that the schemes perform very similar in terms of convergence as long as the number of inner iterations is sufficiently large. This indicates that the inner fixpoint iteration converges very quickly. Moreover, we can see that in this realistic example the number of needed inner iterations $K$ is even lower than the (artificial) examples of Section~\ref{sec:numerics:sharpness} suggest. In particular, the novel scheme performs very well for $K = 3$, which we would not expect to be sufficient for $\omega = 2.8$ according to Figure~\ref{fig:sharpnessComparison}. The numerically observed convergence rate for the proposed method with $K = 2, 3$ as well as for the implicit BDF-2 scheme is around $1.9$ (for small $\tau$). 
%
%
%=============================================================================
%\subsection{Runtime comparison}

The spectra of the matrices of the linear system arising from a BDF-$2$ discretization are very similar to the ones resulting from implicit first-order discretizations. The only difference is the prefactor of $\tau \Bmat$ which does not have big influence on the system. As such, our considerations in choosing preconditioners and solvers for the linear systems are mostly the same as in our previous paper~\cite{AltD24}. Specifically, we use the \textit{conjugate gradient method} for the solution of the decoupled equations. The matrix $\Amat$ is preconditioned using an algebraic multigrid preconditioner~\cite{YanH02}, whereas a Jacobi preconditioner is used for $\Ctau$. The fully implicit system, on the other hand, is solved using \textit{MinRes} with a Schur preconditioner that approximates
\[ 
	\begin{bmatrix} \Amat^{-1} & 0 \\ 0 & \Smat^{-1} \end{bmatrix}
\]
with $\Smat \coloneqq \Dmat^T \Amat \Dmat + \Ctauinv$. We approximate $\Ctauinv$ by its inverse diagonal and $\Amat^{-1}$ and $\Smat^{-1}$ using an algebraic multigrid preconditioner.

Surely, the implicit as well as the semi-explicit method may be further improved. Nevertheless, our implementations seem fair for a runtime comparison. For this, we solve the brain tissue problem once again for different time step sizes and record the corresponding errors and runtimes. The time needed to compute $u^0$ and $p^0$ is not included. The results can be seen in Figure~\ref{fig:brain-error-time}. One can observe that the newly introduced scheme is around four times faster than BDF-$2$, as long as the number of inner iterations $K$ is chosen properly. The second-order fixed stress scheme from Section~\ref{sec:methods:fixedstress} can be faster for lower accuracies, but becomes slower for smaller time steps. 
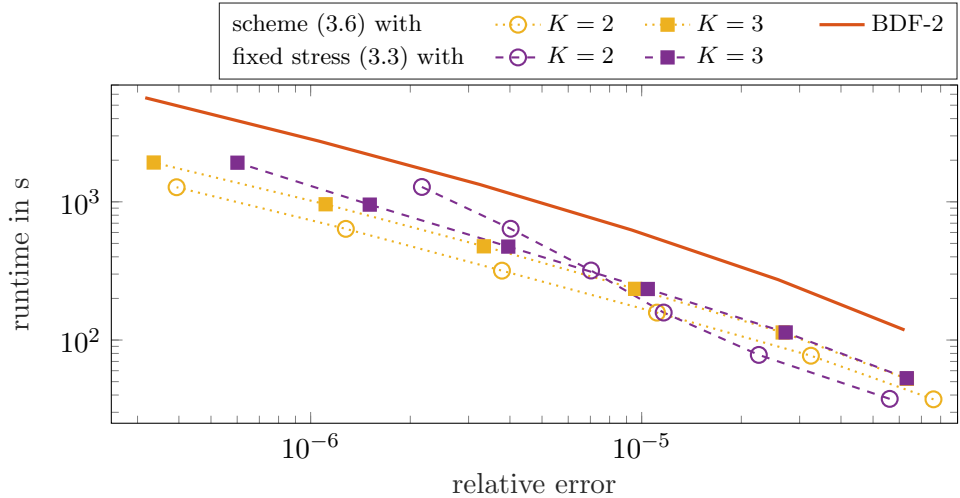
\begin{figure}
%%%%%%% omega = 2.8 %%%%%%%%
\begin{tikzpicture}
		\begin{axis}[%
		width=0.75\textwidth,
		height=0.3\textwidth,
		at={(0\textwidth,0\textwidth)},
		scale only axis,
		unbounded coords=jump,
		xlabel style={font=\color{white!15!black}},
		xlabel={relative error},
		xmode=log,
		ymode=log,
		% yminorticks=true,
		ylabel style={font=\color{white!15!black}},
		ylabel={runtime in s},
		axis background/.style={fill=white},
		legend columns = 4,
		legend style={
			at={(1.0,1.13)},
			anchor=east,
			legend cell align=left,
			align=left,
			draw=white!15!black,
			font=\Small
		},
		xmin = 2.5e-7,
		xmax = 9e-5,
		ymin = 25,
		ymax = 7000,
		% mark repeat=2,
		% mark phase=4,
		mark size=3.0,
		]
		\addlegendimage{empty legend}
		\addlegendentry{scheme~\eqref{eq:scheme} with\quad}

		% \addplot [color = mycolor4, dotted, mark=triangle*, thick, every mark/.append style={solid}] table [col sep=semicolon, y index = {0}, x index = {7}] {experiments/output/error_time_3d.csv};
		% \addlegendentry{$K = 1$\quad}
		\addplot [color = mycolor4, dotted, mark=o, thick, every mark/.append style={solid}] table [col sep=semicolon, y index = {1}, x index = {8}] {experiments/output/error_time_3d.csv};
		\addlegendentry{$K = 2$\quad}
		\addplot [color = mycolor4, dotted, mark=square*, thick, every mark/.append style={solid}, mark size=2.3] table [col sep=semicolon, y index = {2}, x index = {9}] {experiments/output/error_time_3d.csv};
		\addlegendentry{$K = 3$\qquad}

		\addlegendentry{BDF-$2$}
		
		\addplot [color = mycolor2, very thick, every mark/.append style={solid}] table [col sep=semicolon, y index = {3}, x index = {10}] {experiments/output/error_time_3d.csv};

		\addlegendimage{empty legend}
		\addlegendentry{fixed stress~\eqref{eq:fixed-stress} with\quad}

		% \addplot [color = mycolor3, dashed, mark=triangle*, thick, every mark/.append style={solid}] table [col sep=semicolon, y index = {4}, x index = {11}] {experiments/output/error_time_3d.csv};
		% \addlegendentry{$K = 1$\quad}
		\addplot [color = mycolor3, dashed, mark=o, thick, every mark/.append style={solid}] table [col sep=semicolon, y index = {5}, x index = {12}] {experiments/output/error_time_3d.csv};
		\addlegendentry{$K = 2$\quad}
		\addplot [color = mycolor3, dashed, mark=square*, thick, every mark/.append style={solid}, mark size=2.3] table [col sep=semicolon, y index = {6}, x index = {13}] {experiments/output/error_time_3d.csv};
		\addlegendentry{$K = 3$\qquad}
		\end{axis}
		\end{tikzpicture}

		\caption{ \label{fig:brain-error-time}
			Runtime comparison of the novel scheme~\eqref{eq:scheme} and the implicit BDF-$2$ scheme for solving the brain tissue problem of Section~\ref{sec:numerics:brain}.
		}
\end{figure}
%
%
%=============================================================================
%\subsection{Nonlinear example}
%\RA{könnten eigentlich auch mal was nichtlineares testen}
% mention \cite{AltM22}
%
%
%=============================================================================
%=========  Conclusions
%=============================================================================
\section{Conclusions}
Within this paper, we have extended the iterative scheme introduced in~\cite{AltD24} to second order. The scheme combines the semi-explicit approach with an inner iteration, which allows to consider linear poroelasticity with arbitrary coupling strength~$\omega$. Depending on the size of~$\omega$, the number of inner iteration steps~$K$ can be computed a priori, guaranteeing convergence of the scheme. A real-world numerical experiment of brain tissue proves the supremacy of the proposed method. 
%
%
%=============================================================================
%=========  Acknowledgments
%=============================================================================
\section*{Acknowledgments} 
RA acknowledges support by the Deutsche Forschungsgemeinschaft (DFG, German Research Foundation) - 467107679. Moreover, parts of this work were carried out while RA was affiliated with the Institute of Mathematics and the Centre for Advanced Analytics and Predictive Sciences (CAAPS) at the University of Augsburg. 
MD acknowledges support by the Deutsche Forschungsgemeinschaft - 455719484.
%
%
%=============================================================================
%=========  Bibs
%=============================================================================
\bibliographystyle{alpha}
\bibliography{references}
%
%
%=============================================================================
%=========  Appendix
%=============================================================================
\appendix
\section{Proof of a Norm Estimate}\label{app:proofs}
In this appendix, we present %the technical proof showing 
a matrix norm estimate used in Proposition~\ref{prop:convergence-splitting}. 
\begin{lemma}[Norm of the recursion matrix] \label{lem:recursion-matrix}
Under the assumptions of Proposition \ref{prop:convergence-splitting}, including the bound $3\,\omega^K < (2+\omega)^{K-1}$, the matrix 
\[
	\calS \coloneqq
	\begin{bmatrix}
	2\Tmat\Smat^{K-1} + \tfrac13 \mathbf{I} & -\tfrac53\Tmat\Smat^{K-1} & \tfrac13\Tmat\Smat^{K-1} \\
	\mathbf{I} & 0 & 0 \\
	0 & \mathbf{I} & 0
	\end{bmatrix}
	\in \R^{3n_p, 3n_p}
\]
satisfies $\|\calS\|_\Ctau < 1$.
\end{lemma}
\begin{proof}
With the symmetric positive semi-definite matrix~$\tilde \Tmat \coloneqq \Cmat_\tau^{-1/2}\Tmat \Smat^{K-1} \Cmat_\tau^{1/2}$, the norm $\|\calS\|_\Ctau$ coincides with the spectral radius of 
\[
\begin{bmatrix}
2\tilde \Tmat + \tfrac13 \mathbf{I} & -\tfrac53\tilde \Tmat & \tfrac13\tilde \Tmat \\
\mathbf{I} & 0 & 0 \\
0 & \mathbf{I} & 0
\end{bmatrix}. 
\]
To compute the eigenvalues of this block matrix, one can simply do the same for the corresponding 3x3-matrix by replacing the instances of $\tilde \Tmat$ with a variable $\lambda$, i.e., 
\[
	\begin{bmatrix}
	2\lambda + \tfrac13 & -\tfrac53\lambda & \tfrac13\lambda \\
	1 & 0 & 0 \\
	0 & 1 & 0
	\end{bmatrix}.  
\]
This matrix has the eigenvalues $\tfrac13$ and $\lambda \pm \sqrt{\lambda^2 - \lambda}$ for any $\lambda \in \R$ and corresponding eigenvectors
\[
\begin{bmatrix} \tfrac19 \\ \tfrac13 \\ 1 \end{bmatrix}, \qquad
\begin{bmatrix} (\lambda - \sqrt{\lambda^2 - \lambda})^2 \\ \lambda - \sqrt{\lambda^2 - \lambda} \\ 1 \end{bmatrix}, \qquad
\begin{bmatrix} (\lambda + \sqrt{\lambda^2 - \lambda})^2 \\ \lambda + \sqrt{\lambda^2 - \lambda} \\ 1 \end{bmatrix}.
\]
Let $v_1, \dots, v_{n_p}$ be an eigenbasis of $\tilde \Tmat$ with eigenvalues $\lambda_1, \dots, \lambda_{n_p}$. It is then easy to see, that
\[
\begin{bmatrix} \tfrac19 v_i \\ \tfrac13 v_i \\ v_i \end{bmatrix}, \qquad
\begin{bmatrix} (\lambda - \sqrt{\lambda^2 - \lambda})^2 v_i \\ \lambda - \sqrt{\lambda^2 - \lambda} v_i \\ v_i \end{bmatrix}, \qquad
\begin{bmatrix} (\lambda + \sqrt{\lambda^2 - \lambda})^2 v_i \\ \lambda + \sqrt{\lambda^2 - \lambda} v_i \\ v_i \end{bmatrix}
\]
is an eigenbasis of the block matrix with corresponding eigenvalues $\tfrac13$ and $\lambda_i \pm \sqrt{\lambda_i^2 - \lambda_i}$.

We are interested in a criteria for the $\Ctau$-norm of $\calS$ to be smaller than $1$. Since $\tfrac13 < 1$, we only need to check that $|\lambda_i \pm \sqrt{\lambda_i^2 - \lambda}| < 1$ for all $i$. It is easy to see that this cannot be the case if $\lambda_i \le -1$ or $\lambda_i \ge 1$, so we will assume $|\lambda_i| < 1$. The case $\lambda_i = 0$ is trivial, so we will only consider $\lambda_i < 0$ and $\lambda_i > 0$. For the former the term inside the square root is positive, so we can calculate
\begin{align*}
\big|\lambda_i \pm \sqrt{\lambda_i^2 - \lambda_i}\big| < 1
&\iff \lambda_i - \sqrt{\lambda_i^2 - \lambda_i} > -1
\iff 1 + \lambda_i > \sqrt{\lambda_i^2 - \lambda_i} \\
&\iff (1 + \lambda_i)^2 > \lambda_i^2 - \lambda_i
\iff 1 + 2\lambda_i + \lambda_i^2 > \lambda_i^2 - \lambda_i \\
&\iff \lambda_i > -\tfrac13.
\end{align*}
Otherwise if $\lambda_i > 0$, due to $\lambda_i < 1$ the term inside the square root is negative and we get
\begin{align*}
\big|\lambda_i \pm \sqrt{\lambda_i^2 - \lambda_i}\big| < 1
&\iff \big|\lambda_i \pm \sqrt{\lambda_i^2 - \lambda_i} \big|^2 < 1
\iff \lambda_i^2 + (\lambda_i - \lambda^2) < 1 \\
&\iff \lambda_i < 1.
\end{align*}
Hence, the desired estimate holds if the spectral norm of $\tilde \Tmat$ is bounded by $1/3$. For the similar matrix $\Tmat \Smat^{K-1}$, the derivations in Lemma~\ref{lem:matrices} show 
\[
	\rho(\Tmat\Smat^{K-1})
	\le \omega(1-\gamma)^{K-1},
\]
leading to the sufficient condition $ \omega^K / (2+\omega)^{K-1} = \omega(1-\gamma)^{K-1} < \frac13$. 
\end{proof}
\end{document}